\documentclass[12pt]{amsart}

\usepackage{geometry} 
\usepackage{amsmath,amsthm,amssymb,amscd,tikz,tikz-cd, mathrsfs, mathtools, manfnt, hyperref}
\usepackage[all, cmtip]{xy}

\newcommand{\N}{\mathbb N}
\newcommand{\Z}{\mathbb Z}

\newcommand{\F}{\mathbb F}

\newcommand{\D}{\mathbb D}
\newcommand{\E}{\mathbb E}

\newcommand{\bS}{\mathbb S}

\newcommand{\ie}{i.e.\ }
\newcommand{\eg}{e.g.\ }

\newcommand{\inv}{^{-1}}

\newcommand{\ur}{\text{\smash{\raisebox{-1ex}{\scalebox{1.5}{$\urcorner$}}}}}
\renewcommand{\ll}{\text{\smash{{\scalebox{1.5}{$\llcorner$}}}}}
\DeclareMathOperator{\dia}{dia}

\DeclareMathOperator{\Sp}{Sp}
\DeclareMathOperator{\St}{St}
\DeclareMathOperator{\loc}{loc}
\DeclareMathOperator{\stab}{stab}
\newcommand{\op}{^\text{op}}
\DeclareMathOperator{\Cyl}{Cyl}

\newcommand{\swtrans}{\mathbin{\rotatebox[origin=c]{225}{$\Rightarrow$}}}
\newcommand{\netrans}{\mathbin{\rotatebox[origin=c]{45}{$\Rightarrow$}}}

\newcommand{\cC}{\mathcal C}
\newcommand{\cE}{\mathcal E}
\newcommand{\cD}{\mathcal D}

\newcommand{\cS}{\mathcal S}

\newcommand{\Cat}{\mathbf{Cat}}
\newcommand{\CAT}{\mathbf{CAT}}

\DeclareMathOperator{\Hom}{Hom}

\DeclareMathOperator{\id}{id}

\newcommand{\Dia}{\mathbf{Dia}}
\DeclareMathOperator{\Ho}{Ho}

\renewcommand{\phi}{\varphi}

\numberwithin{equation}{section}
\numberwithin{figure}{section}

\theoremstyle{definition}
\newtheorem{theorem}[equation]{Theorem}
\newtheorem{prop}[equation]{Proposition}
\newtheorem{defn}[equation]{Definition}
\newtheorem{lemma}[equation]{Lemma}
\newtheorem{remark}[equation]{Remark}
\newtheorem{cor}[equation]{Corollary}

\newtheorem{notn}[equation]{Notation}
\newtheorem{cons}[equation]{Construction}

\newtheorem{innercustomthm}{Main Theorem}
\newenvironment{thm2}[1]{\renewcommand\theinnercustomthm{#1}\innercustomthm}{\endinnercustomthm}

\begin{document}

\title{Stabilization of derivators revisited}
\author{Ian Coley}
\address{Department of Mathematics, University of California, Los Angeles}
\urladdr{\href{http://iancoley.org}{iancoley.org}}
\email{\href{mailto:iacoley@math.ucla.edu}{iacoley@math.ucla.edu}}
\begin{abstract}
We revisit and improve Alex Heller's results on the stabilization of derivators in \cite{Hel97}, recovering his results entirely. Along the way we give some details of the localization theory of derivators and prove some new results in that vein.
\end{abstract}

\maketitle
\tableofcontents

\section{Introduction}

The theory of derivators was first developed by Heller in \cite{Hel88} and Grothendieck in \cite{Gro90}. Heller used the now-defunct terminology of `hypercategories' and `homotopy theories' to describe the objects of his study. Nonetheless, his and Grothendieck's axioms are essentially the same, and Heller's results in his monograph \emph{Homotopy Theories} \cite{Hel88} and his later paper \emph{Stable homotopy theories and stabilization} \cite{Hel97} are often cited in the modern study of derivators. For instance, Tabuada in \cite{Tab08} uses Heller's stabilization to construct the universal localizing and additive invariants of dg-categories, which are connected to K-theory and non-commutative motives.

The paper came about while trying to answer the following open problem posed by Muro and Raptis in \cite[\S6.3]{MurRap17}:
\begin{quote}
We do not know whether derivator K-theory [...] is invariant under an appropriately defined notion of stabilization which would produce a triangulated derivator.
\end{quote}
Understanding the mechanics of the stabilization that Heller developed in \cite{Hel97} is a necessary first step. A careful reading of that paper will find that, though his main propositions and theorems feel true, Heller rarely gives more than a cursory proof. In the present paper, we provide full proofs and correct certain errors along the way. Further, we remove the hypothesis that the derivators to be stabilized are strong (Definition~\ref{defn:strong}), and so obtain results in greater generality.

The goals and results of this paper are essentially the same as those of \cite{Hel97}. Our first main result deals with localization of pointed derivators, a key ingredient in the proof of stabilization. Heller gives a bare sketch of the proof at \cite[Proposition~7.4]{Hel97} which we complete into a full argument.
\begin{thm2}{1}[Theorem~\ref{thm:vanishingloc}]
Let $\D$ be a pointed derivator and $i\colon A\to B$ be a full subcategory. Then the inclusion of the vanishing subprederivator $\D^{B,A}\to \D^B$ admits a left adjoint. In particular, $\D^{B,A}$ is again a pointed derivator.
\end{thm2}

In order to define the stabilization of a derivator, we will use infinite diagram categories and various maps between them. Heller posits the existence of a functor that does not exist, however, and only assumes the existence of others. In order to complete his argument, we have with total specificity constructed new functors between infinite diagrams, for example at Remark~\ref{rk:wrongw} and Construction~\ref{cons:coherentdia}. These proofs offer a way to move from a heuristic to a rigorous construction.

The final objective is a 2-functorial statement of stabilization.
\begin{thm2}{2}[Theorem~\ref{thm:main2}]
There is a pseudofunctor $\St\colon\textbf{Der}_!\to\textbf{StDer}_!$ associating to any regular pointed derivator $\D$ a stable derivator $\St\D$. This\linebreak pseudofunctor is a localization of 2-categories, \ie it admits a fully faithful right adjoint. In particular, for any stable derivator $\bS$, there is an equivalence of categories of cocontinuous morphisms
\begin{equation*}
\stab^\ast\colon\Hom_!(\St\D,\bS)\to\Hom_!(\D,\bS)
\end{equation*}
induced by precomposition with a cocontinuous morphism $\stab\colon\D\to\St\D$.
\end{thm2}
A key step in this theorem is proving the pseudofunctoriality of $\St$ on morphisms of derivators. Lemma~\ref{lemma:stableequivalence} is our new solution to a major gap in Heller's proof, which is discussed above that lemma. 

There are other theories of stabilization that may be preferred to using derivators. Lurie in \cite[\S1.4]{Lur16} describes stabilization for a huge class of $\infty$-categories, with stronger results for $\infty$-categories admitting nicer properties. His result that most resembles ours is the following:\pagebreak
\begin{theorem}\cite[Corollary~1.4.4.5]{Lur16}
Let $\cC$ and $\cD$ be presentable $\infty$-categories, and suppose that $\cD$ is stable. Then precomposition with $\Sigma_+^\infty\colon\cC\to\Sp(\cC)$ induces an equivalence of $\infty$-categories of left adjoint functors
\begin{equation*}
\operatorname{Fun}^\text{L}(\Sp(\cC),\cD)\to\operatorname{Fun}^\text{L}(\cC,\cD).
\end{equation*}
Here, $\Sigma_+^\infty$ is Lurie's stabilization functor.
\end{theorem}

Hovey in \cite{Hov01} develops the theory of stabilizing a model category $\cC$ at any left Quillen endofunctor $G$. He too establishes a universal property, so long as $\cC$ is a left proper cellular model category. To give the necessary notation to state the property, we define a map of pairs $(\cC,G)\to(\cD, H)$ to be a left Quillen functor $\Phi\colon\cC\to\cD$ with a natural transformation $\tau\colon \Phi G\to H\Phi$ such that $\tau_A$ is a weak equivalence for cofibrant objects $A$.
\begin{theorem}\cite[Corollary~5.4]{Hov01}
Suppose that $(\Phi,\tau)\colon (\cC,G)\to (\cD,H)$ is a map of pairs such that $H$ is a Quillen equivalence. Then there is a functor $\widetilde\Phi\colon\Ho\Sp^\N(\cC,G)\to\Ho\cD$ such that
\begin{equation*}
\vcenter{\xymatrix@C=2em@R=1em{
\Ho\cC\ar[rr]^-{L\Phi}\ar[dr]_-{LF_0}&&\Ho\cD\\
&\Ho\Sp^\N(\cC,G)\ar[ur]_-{\widetilde\Phi}
}}
\end{equation*}
commutes. In this case, $F_0\colon \cC\to\Sp^\N(\cC,G)$ is the (left Quillen) stabilization functor and $LF_0$ its left derived functor.
\end{theorem}

Lurie and Hovey obviously differ in their approaches, but recover similar results when we let the left Quillen endofunctor $G$ be suspension.

We should give some support for the theory of derivators over $\infty$-categories or model categories. Using derivators, we obtain a result that subsumes Lurie's (and Hovey's for $G=\Sigma$) without appealing to $\infty$-theoretic techniques. In particular, any (locally) presentable $\infty$-category gives rise to a derivator that our machinery can stabilize, and our universal property agrees with Lurie's in the presence of an adjoint functor theorem.

Furthermore, we characterize stabilization as a localization of a certain 2-category of derivators, which gives us a functoriality that does not exist for model categories. Lastly, we have an explicit formula for the stabilization morphism (Definition~\ref{defn:stabdefn}), which is lacking for $\infty$-categories in many cases. Stabilization in derivators with respect to other endomorphisms is also under investigation by Groth and Shulman in \cite{GroShu17}, who use the framework of enriched derivators and weighted (co)limits to obtain some nice results for arbitrary derivators. A sequel to the latter paper entitled \emph{Abstract stabilization: the universal absolute} is in preparation.

To give an outline: we first recall the theory of derivators as necessary to understand this paper in \S2. In \S3 we collect some results on the localization theory of derivators and adapt some categorical propositions to the derivator context. We then apply these results to a general theory of `vanishing subderivators' in \S4.  We study in \S5 a particular vanishing subderivator that we use as an intermediate step towards stabilization. We establish the actual construction of the stabilization morphism in \S6, and we conclude with its universal property in \S7.

The author would like to thank Paul Balmer, Moritz Groth, and the UCLA Derivators Seminar for their support. Most of all, the author would like to thank the University of Seville for hosting him as a visiting researcher and Fernando Muro for spending a tireless week hammering out errors in the proof of the universal property.

\section{Preliminaries on (pre)derivators}

A good first reading on derivators is \cite{Gro13}, to which we will refer repeatedly throughout. While our results apply broadly, our proof methodology has a diagrammatic flavor that is emblematic of the theory of derivators.

\begin{defn}
A \emph{prederivator} is a strict 2-functor $\D\colon\Cat\op\to\CAT$. Specifically, to each small category $K\in\Cat$ we assign a (not necessarily small) category\linebreak $\D(K)\in\CAT$, to each functor $u\colon J\to K$ a functor $u^\ast\colon\D(K)\to\D(J)$, and to each natural transformation $\alpha\colon u\to v$ a natural transformation $\alpha^\ast\colon u^\ast\to v^\ast$.
\end{defn}

\begin{defn}
A \emph{derivator} is a prederivator $\D$ satisfying the following four axioms:
\begin{enumerate}
\item[(Der1)] $\D$ sends coproducts to products, \ie the natural map
\begin{equation*}
\D\left(\coprod_{i\in I}K_i\right)\longrightarrow\prod_{i\in I}\D(K_i).
\end{equation*}
is an equivalence of categories for any set $I$ and categories $K_i$. As a special case, $\D(\varnothing)\simeq e$, where $\varnothing$ is the initial (empty) category and $e$ is the final category.

\item[(Der2)] A morphism $f\colon X\to Y$ in $\D(K)$ is an isomorphism if and only if \linebreak$k^\ast f\colon k^\ast X\to k^\ast Y$ is an isomorphism in $\D(e)$ for the functors $k\colon e\to K$ \linebreak classifying each object $k\in K$.

\item[(Der3)] For every functor $u\colon J\to K$, $u^\ast\colon \D(K)\to\D(J)$ has a left adjoint \linebreak$u_!\colon\D(J)\to\D(K)$, called the \emph{left (homotopy) Kan extension} of $u$ and a right adjoint $u_\ast\colon\D(J)\to\D(K)$, called the \emph{right (homotopy) Kan extension} of $u$.

\item[(Der4)] For every functor $u\colon J\to K$ and any object $k\in K$, if we identify $k$ with $k\colon e\to K$, we have a natural transformation $\alpha$ given by
\begin{equation*}
\xymatrix{
J_{/k}\ar[r]^{\text{pr}}\ar[d]_{\pi\,=\,\pi_{J_{/k}}}&J\ar[d]^u\ar@{}[dl]|{\swtrans}\ar@{}[dl]<-1.2ex>|{\alpha}\\
e\ar[r]_k&K
}
\end{equation*}
The category $J_{/k}$ is the slice category, and the functor $\text{pr}$ is the forgetful functor, where an object $(j,f\colon u(j)\to k)\in J_{/k}$ is sent to $j\in J$. The calculus of mates (see, for example, \cite[Appendix A]{GroPonShu14} or \cite{KelStr74}) yields
\begin{equation*}
\xymatrix{
\D(e)&\ar[l]_-{\pi_!}\D(J_{/k})\ar@{}[dl]|{\swtrans}&\D(J)\ar@{}[dl]|{\swtrans}\ar@{}[dl]<-1.2ex>|{\alpha^\ast}\ar[l]_-{\text{pr}^\ast}&{}\ar@{}[dl]|{\swtrans}\\
{}&\D(e)\ar@(l,d)[ul]^{=}\ar[u]^-{\pi^\ast}&\D(K)\ar[l]^-{k^\ast}\ar[u]_-{u^\ast}&\D(J)\ar[l]^-{u_!}\ar@(u,r)[ul]_{=}
}
\end{equation*}
Combining these transformations together, we obtain a natural transformation
\begin{equation*}
\alpha_!\colon\pi_!\text{pr}^\ast \longrightarrow k^\ast u_!
\end{equation*}
We require this to be a natural isomorphism for any functor $u$ and any $k\in K$. Dually, using the slice category $J_{k/}$, we require
\begin{equation*}
\alpha_\ast\colon k^\ast u_\ast \longrightarrow \pi_\ast\text{pr}^\ast 
\end{equation*}
to be a natural isomorphism for any functor $u$ and any $k\in K$.
\end{enumerate}
\end{defn}

There is a `fifth axiom' for derivators which is not necessary in all contexts, ours included. We mention it, however, because it is present in Heller's original paper and important for his arguments.

Let $[1]$ be the arrow category, that is, the diagram $0\to 1$. We have two functors $s,t\colon e\to[1]$ classifying the source and target and a (unique) natural transformation $\alpha\colon s\to t$. These data assemble to, for any prederivator $\D$, a functor $\dia_{[1]}\colon\D([1])\to~\D(e)^{[1]}$, which we call the \emph{underlying diagram functor}. In words, it takes a `coherent' arrow-shaped diagram and sends it to an `incoherent' arrow. There is a similar construction of the \emph{partial underlying diagram functor} $\dia_{[1],K}\colon\D([1]\times K)\to\D(K)^{[1]}$ by leaving the $K$ dimension of the diagram coherent.
\begin{defn}\label{defn:strong}
(Der5) We say a prederivator is \emph{strong} if for any category $K\in \Cat$, the functor $\dia_{[1],K}:\D([1]\times K)\to\D(K)^{[1]}$ is full and essentially surjective. 
\end{defn}

A \emph{morphism of (pre)derivators} $\Phi\colon\D\to\E$ is a pseudonatural transformation of 2-functors. In particular, for each $K\in\Cat$ we have a functor $\Phi_K\colon\D(K)\to\E(K)$ and for every $u\colon J\to K$ we have a natural isomorphism $\gamma_u^\Phi\colon u^\ast \Phi_K\to \Phi_J u^\ast$
\begin{equation*}
\xymatrix{
\D(K)\ar[r]^{\Phi_K}\ar[d]_-{u^\ast}&\E(K)\ar[d]^-{u^\ast}\ar@{}[dl]|{\swtrans}\ar@{}@<-1.75ex>[dl]|{\gamma_u^\Phi}\\
\D(J)\ar[r]_{\Phi_J}&\E(J)
}
\end{equation*}
The family $\{\gamma_u^\Phi\}$ is subject to coherence conditions. Foremost, for $u\colon J\to K$ and $v\colon I\to J$ we require that the pasting on the left be equal to the square on the right:
\begin{equation*}
\vcenter{\xymatrix{
\D(K)\ar[r]^{\Phi_K}\ar[d]_-{u^\ast}&\E(K)\ar[d]^-{u^\ast}\ar@{}[dl]|{\swtrans}\ar@{}@<-1.75ex>[dl]|{\gamma_u^\Phi}\\
\D(J)\ar[d]_-{v^\ast}\ar[r]^{\Phi_J}&\E(J)\ar[d]^-{v^\ast}\ar@{}[dl]|{\swtrans}\ar@{}@<-1.75ex>[dl]|{\gamma_v^\Phi}\\
\D(I)\ar[r]_{\Phi_I}&\E(I)
}}
\quad=\quad
\vcenter{\xymatrix{
\D(K)\ar[r]^{\Phi_K}\ar[d]_-{(uv)^\ast}&\E(K)\ar[d]^-{(uv)^\ast}\ar@{}[dl]|{\swtrans}\ar@{}@<-1.75ex>[dl]|{\gamma_{uv}^\Phi}\\
\D(I)\ar[r]_{\Phi_I}&\E(I)
}}
\end{equation*}
In addition, we require $\gamma_{\id_J}^\Phi=\id_{\D(J)}$. There is also compatibility with natural transformations in $\Cat$. For two functors $u,v\colon J\to K$ and a natural transformation $\alpha\colon u\to v$, we require the below pastings to be equal:
\begin{equation*}
\vcenter{\xymatrix{
\D(K)\ar[r]^{\Phi_K}\ar[d]^-{u^\ast}="S0"\ar@(l,l)[d]_-{v^\ast}="T0"&\E(K)\ar[d]^-{u^\ast}\ar@{}[dl]|{\swtrans}\ar@{}@<-1.75ex>[dl]|{\gamma_u^\Phi}\\
\D(J)\ar[r]_{\Phi_J}&\E(J)
\ar@{}"S0";"T0" |{\mathbin{\rotatebox[origin=c]{180}{$\Rightarrow$}}}_-{\alpha^\ast}
}}
\quad=\quad
\vcenter{\xymatrix{
\D(K)\ar[r]^{\Phi_K}\ar[d]_-{v^\ast}&\E(K)\ar[d]_-{v^\ast}="T0"\ar@(r,r)[d]^-{u^\ast}="S0"\ar@{}[dl]|{\swtrans}\ar@{}@<-1.75ex>[dl]|{\gamma_v^\Phi}\\
\D(J)\ar[r]_{\Phi_J}&\E(J)
\ar@{}"S0";"T0" |{\mathbin{\rotatebox[origin=c]{180}{$\Rightarrow$}}}_-{\alpha^\ast}
}}
\end{equation*}
Details can be found in \cite[Definition~7.5.2]{Bor94b}. A morphism $\Phi$ is an \emph{equivalence of derivators} if $\Phi_K$ is an equivalence of categories for every $K\in\Cat$.

If $\Phi,\Psi\colon\D\to \E$ are two morphisms of derivators, a natural transformation \linebreak$\rho\colon\Phi\to\Psi$ is given by a \emph{modification} of pseudonatural transformations of 2-functors. This is a natural transformation $\rho_K\colon\Phi_K\to\Psi_K$ for every $K\in\Cat$ satisfying the following coherence condition: if $u,v\colon J\to K$ are two functors and $\alpha\colon u\to v$ is a natural transformation, then we have an equality of pastings
\begin{equation*}
\vcenter{\xymatrix@C=3em{
\D(K)\ar@/^1pc/[r]^-{u^\ast}="S1"\ar@/_1pc/[r]_-{v^\ast}="T1"&\D(J)\ar@/^1pc/[r]^-{\Phi_J}="S0"\ar@/_1pc/[r]_-{\Psi_J}="T0"&\E(J)
\ar@{} "S0";"T0" |{\mathbin{\rotatebox[origin=c]{270}{$\Rightarrow$}}} \ar@{}@<-1.4ex>"S0";"T0"|{\rho_J}
\ar@{} "S1";"T1" |{\mathbin{\rotatebox[origin=c]{270}{$\Rightarrow$}}} \ar@{}@<-1.4ex>"S1";"T1"|{\alpha^\ast}
}}
\quad=\quad
\vcenter{\xymatrix@C=3em{
\D(K)\ar@/^1pc/[r]^-{\Phi_K}="S0"\ar@/_1pc/[r]_-{\Psi_K}="T0"&\E(K)\ar@/^1pc/[r]^-{u^\ast}="S1" \ar@/_1pc/[r]_-{v^\ast}="T1"&\E(J)
\ar@{} "S0";"T0" |{\mathbin{\rotatebox[origin=c]{270}{$\Rightarrow$}}} \ar@{}@<-1.4ex>"S0";"T0"|{\rho_K}
\ar@{} "S1";"T1" |{\mathbin{\rotatebox[origin=c]{270}{$\Rightarrow$}}} \ar@{}@<-1.4ex>"S1";"T1"|{\alpha^\ast}
}}
\end{equation*}
See also \cite[Definition~7.5.3]{Bor94b}. A modification $\rho$ is called an \emph{isomodification} if $\rho_K$ is a natural isomorphism for every $K\in\Cat$.

\begin{notn}
For two prederivators $\D$ and $\E$, we let $\Hom(\D,\E)$ denote the category with objects morphisms of prederivators and maps modifications. This gives us a 2-category $\mathbf{PDer}$ and a full 2-subcategory $\mathbf{Der}$ consisting of all derivators.
\end{notn}

A ready source for morphisms of derivators comes from functors in $\Cat$. Let $\D$ be a prederivator, and let $u\colon J\to K$ be a functor. We denote by $\D^A$ the `shifted' prederivator taking $J\mapsto \D(J\times A)$ and $u\mapsto (u\times \id_A)^\ast\colon\D(K\times A)\to\D(J\times A)$. If $\D$ has other nice properties (\eg being a derivator), so will $\D^A$. See \cite[Theorem~1.25]{Gro13} or \cite[Proposition~4.3]{Gro13} for two examples.

We can think of the functor $u^\ast\colon \D(K)\to\D(J)$ instead as a morphism of the corresponding shifted derivators $u^\ast\colon\D^K\to\D^J$. The same is true for $u_!,u_\ast$ when they exist. These will comprise the majority of morphisms appearing in our constructions.

We refer to $\D(e)$ as the \emph{base} of the derivator and think of $\D(K)$ as coherent\linebreak $K$-shaped diagrams with objects and morphisms in $\D(e)$. Sometimes we call $\D(K)$ the \emph{levels} of the derivator. With this in mind, we will use the notation $X\in\D$ to refer to $X\in\D(K)$ for some $K\in\Cat$ when working with morphisms of derivators. For example, if we have a morphism $\Phi\colon\D\to\E$ of derivators and we wish to prove something about $\Phi_KX$ for all $K\in\Cat$ and all $X\in\D(K)$, it will be easier to write $\Phi X$ for $X\in\D$, leaving `at any category $K$' or `levelwise' understood.

\begin{remark}
It is not strictly necessary that we take as domain of our derivators all of $\Cat$. A list of axioms for suitable (sub)categories of diagrams $\Dia\subset\Cat$ can be found at \cite[Definition~1.12]{Gro13} or \cite[p.3]{Mal07}.

A common choice for $\Dia$ is $\mathbf{Dir_f}$, the 2-category of finite direct categories, \ie categories whose nerve has finitely many nondegenerate simplices. Keller in \cite{Kel07} gives a construction of a triangulated derivator with domain $\mathbf{Dir_f}$ for any exact category. Cisinski in \cite[Th\'eor\`eme~2.21]{Cis10} proves that a Waldhausen category whose weak equivalences satisfy some mild properties gives rise to a left derivator with domain $\mathbf{Dir_f}$, \ie one which admits only left Kan extensions $u_!$ and satisfies the corresponding half of Der4.

However, we will \emph{not} be able to study these sorts of derivators directly. Heller's (and thus our) approach uses infinite posets that are not diagrams in $\mathbf{Dir_f}$.
\end{remark}

Fix a domain $\Dia$ that contains all diagram shapes that appear in the rest of this paper, for instance the 2-subcategory of all (possibly infinite) posets $\mathbf{Pos}$ or $\Cat$.\pagebreak

\section{Localization of derivators}

We first need to recall two special classes of morphisms of derivators.

\begin{defn}
Let $u\colon J\to K$ be a functor in $\Dia$. A morphism $\Phi\colon\D\to\E$ \emph{preserves left Kan extensions along $u$} if the natural transformation $u_!\Phi_J\to\Phi_K u_!$ given by the below pasting
\begin{equation*}
\xymatrix{
\D(J)\ar@(d,l)[dr]_-=\ar[r]^-{u_!}&\D(K)\ar@{}[dl]|{\netrans}\ar[r]^{\Phi_K}\ar[d]_-{u^\ast}&\E(K)\ar[d]^-{u^\ast}\ar@{}[dl]|{\netrans}\ar@{}@<-1.75ex>[dl]|{(\gamma_u^\Phi)\inv}\ar@(r,u)[dr]^-=&\ar@{}[dl]|{\netrans}\\
{}&\D(J)\ar[r]_{\Phi_J}&\E(J)\ar[r]_{u_!}&\E(K)
}
\end{equation*}
is an isomorphism. A morphism $\Phi$ is called \emph{cocontinuous} if it preserves left Kan extensions along every $u\colon J\to K$ in $\Dia$.

Dually, $\Phi$ \emph{preserves right Kan extensions along $u$} if the natural transformation $\Phi_K u_\ast\to u_\ast \Phi_J$ given by
\begin{equation*}
\xymatrix{
\D(J)\ar@(d,l)[dr]_-=\ar[r]^-{u_\ast}&\D(K)\ar@{}[dl]|{\swtrans}\ar[r]^{\Phi_K}\ar[d]_-{u^\ast}&\E(K)\ar[d]^-{u^\ast}\ar@{}[dl]|{\swtrans}\ar@{}@<-1.75ex>[dl]|{\gamma_u^\Phi}\ar@(r,u)[dr]^-=&\ar@{}[dl]|{\swtrans}\\
{}&\D(J)\ar[r]_{\Phi_J}&\E(J)\ar[r]_{u_\ast}&\E(K)
}
\end{equation*}
is an isomorphism. A morphism $\Phi$ is called \emph{continuous} if it preserves right Kan extensions along every $u\colon J\to K$ in $\Dia$.
\end{defn}

There is a ready source of continuous and cocontinuous morphisms of derivators.

\begin{defn}
Let $\Phi\colon\D\to\E$ and $\Psi\colon\E\to\D$ be two morphisms of derivators. We say that $\Phi$~is \emph{left adjoint} to $\Psi$ (equivalently, $\Psi$ is \emph{right adjoint} to $\Phi$) if there exist two modifications $\eta\colon \id_\D\to\Psi\Phi$ and $\varepsilon\colon \Phi\Psi\to\id_\E$ satisfying the usual triangle identities.
\end{defn}
At each $K\in\Cat$ we have an adjunction of categories $(\Phi_K,\Psi_K)$, but the units and counits of these adjunctions must assemble to modifications. Unlike in the case of equivalences, a morphism $\Phi\colon\D\to\E$ that admits levelwise right adjoints may not be a left adjoint morphism of derivators. In fact, this is the case if any only if $\Phi$ is cocontinuous; see \cite[Proposition~2.9]{Gro13}. Dually, a morphism $\Psi$ admitting levelwise left adjoints is itself a right adjoint morphism of derivators if and only if it is continuous.

In particular, for $u\colon J\to K$, the morphisms $u^\ast\colon\D^K\to\D^J$ and $u_!\colon\D^J\to\D^K$ are left adjoint morphisms, hence cocontinuous. Dually, the morphisms $u^\ast\colon\D^K\to~\D^J$ and $u_\ast\colon\D^J\to\D^K$ are right adjoint morphisms, hence continuous. For certain $u\colon J\to K$ and suitable $\D$, the morphism $u_\ast\colon\D^J\to\D^K$ is cocontinuous as well, as we will see in Definition~\ref{defn:extraadjoint}.

We now turn to the localization theory of derivators. There is no one resource for this topic at the moment, but we will try to summarize the theory as it relates to our ultimate goal of stabilization. For a general reference on localizations of categories, see \cite[\S1]{GabZis67} or \cite{Kra10}.

Let $\cC$ be an ordinary category. If we have some class of morphisms $S\subset\cC^{[1]}$ that we would like to invert, we can ask whether there is a category $\cC[S\inv]$ and functor $L_S\colon\cC\to\cC[S\inv]$ inverting $S$ and admitting a fully faithful left or right adjoint. The localization functor $L_S$ and the category $\cC[S\inv]$ are essentially unique and satisfy a universal property in $\CAT$. There are more specialized localization theories in, for example, model categories or triangulated categories.

However, we are going to work with a broad class of derivators so will not necessarily have more refined machinery. Moreover, we will care less about starting with a class of morphisms in $\D(e)$, but will have some subcategory of $\D(e)$ which we would like to be reflective or coreflective. If this is the case, we obtain a (co)localization of $\D(e)$ onto that subcategory and can compute which morphisms have been inverted. The former approach is the subject of an upcoming paper of Ioannis Lagkas, which we will not address here.

Let $\D$ be a derivator. A prederivator $\E$ is called a \emph{full subprederivator} of $\D$ if there is a morphism $\iota\colon\E\to\D$ which is levelwise fully faithful. There is no reason for $\D$ to reflect any of its properties onto $\E$, but there is a straightforward criterion.

\begin{lemma}\label{lemma:localisederivator}
\cite[Lemme~4.2]{Cis08} Let $\E$ be a full subprederivator of a derivator~$\D$. Assume that the morphism $\iota\colon\E\to\D$ admits either a left adjoint $L$ or right adjoint $R$. Then $\E$ is also a derivator.
\end{lemma}

If $\iota$ admits a left adjoint, we call $\E$ a \emph{localization} of $\D$, and if $\iota$ admits a right adjoint we call it a \emph{colocalization}. We give the proof of this lemma because Cisinski's result uses terminology different from ours, contains several typographical errors, and is in French.

\begin{proof}
Let us prove Der1 and Der2 first, which do not depend on which side the adjoint is. For Der1, consider the following commutative diagram for $K_i\in\Dia$:
\begin{equation*}
\vcenter{\xymatrix@C=4em{\displaystyle
\E\left(\displaystyle\coprod_{i\in I} K_i\right)\ar[r]^-{\iota_{\coprod K_i}}\ar[d]_-{\prod j_i^\ast}&\D\left(\displaystyle\coprod_{i\in I}K_i\right)\ar@{}[dl]|\swtrans\ar[d]_-{\prod j_i^\ast}\ar[r]^{A_{\coprod K_i}}&\E\left(\displaystyle\coprod_{i\in I} K_i\right)\ar[d]^-{\prod j_i^\ast}\ar@{}[dl]|\swtrans\\
\displaystyle\prod_{i\in I}\E(K_i)\ar[r]_-{\prod \iota_{K_i}}&\displaystyle\prod_{i\in I}\D(K_i)\ar[r]_-{\prod A_{K_i}}& \displaystyle\prod_{i\in I}\E(K_i)
}}
\end{equation*}
where $A$ is the adjoint to $\iota$. Because $\D$ satisfies Der1, the middle vertical functor is an equivalence of categories. We would like to prove that the other vertical functor is also an equivalence, so we will show it is fully faithful and essentially surjective.

First we examine the lefthand square. The functor $\iota_{\coprod K_i}$ is fully faithful by assumption, hence the top composition $\prod j_i^\ast\circ\iota_{\coprod K_i}$ is fully faithful. Thus the bottom composition\linebreak $\prod \iota_{K_i}\circ \prod j_i^\ast$ is also fully faithful, and by assumption $\prod \iota_{K_i}$ is fully faithful. For any $X,Y\in \E(\coprod K_i)$, we have isomorphisms
\begin{equation*}
\vcenter{\xymatrix@C=0em{
&\Hom_{\prod\E(K_i)}(\prod j_i^\ast X, \prod j_i^\ast Y)\ar[dd]_-\cong^-{\prod \iota_{K_i}}\\
\Hom_{\E(\coprod K_i)}(X,Y)\ar[ur]^(0.4){\prod j_i^\ast}\ar[dr]^-\cong\\
&\Hom_{\prod\D(K_i)}(\prod \iota_{K_i}(\prod j_i^\ast X),\prod \iota_{K_i}( \prod j_i^\ast Y))
}}
\end{equation*}
which proves that $\prod j_i^\ast$ induces an isomorphism on hom-sets, \ie is fully faithful.

Next we examine the righthand square. Whether $A$ is the left or right adjoint to $\iota$, it is levelwise essentially surjective. For $A=L$, because $\iota_K$ is fully faithful, the counit of the adjunction $L_K\iota_K\Rightarrow \id_{\E(K)}$ is an isomorphism. Hence for any $X\in \E(K)$, we have $X\cong L_K(\iota_KX)$, so $X$ is in the essential image of $L_K$. For $A=R$, the unit of the adjunction is an isomorphism and the same result follows.

The bottom composition $\prod A_{K_i}\circ \prod j_i^\ast$ is thus essentially surjective, which implies the top composition $\prod j_i^\ast \circ A_{\coprod K_i}$ is too. Thus for any $Z\in \prod\E(K_i)$, there exists some $X\in \D(\coprod K_i)$ such that $\prod j_i^\ast (A_{\coprod K_i} X)\cong Z$. In particular, the object\linebreak $A_{\coprod K_i} X=Y\in \E(\coprod K_i)$ satisfies $\prod j_i^\ast Y\cong Z$, so that $\prod j_i^\ast$ is essentially surjective as well. Having shown that $\prod j_i^\ast$ is both fully faithful and essentially surjective, it is an equivalence of categories and thus $\E$ satisfies Der1.

For Der2, suppose that $f\colon X\to Y$ is a map in $\E(K)$. We need to show that $f$ is an isomorphism if and only if $k^\ast f\colon k^\ast X\to k^\ast Y$ is an isomorphism for all $k\in K$. The image $\iota_K f$ in $\D(K)$ is an isomorphism if and only if $f$ is an isomorphism as $\iota_K$ is fully faithful. Further, $k^\ast \iota_K f \cong \iota_e k^\ast f$ as $\iota$ is a morphism of (pre)derivators, and $\iota_e k^\ast f$ is an isomorphism if and only if $k^\ast f$ is an isomorphism. Putting all these implications together, we obtain Der2 for $\E$.

Now we prove that $\E$ admits left and right Kan extensions, and at this point we assume $A=L$ is the left adjoint, the dual case being similar but not identical. Let $u\colon J\to K$ be a functor in $\Dia$, $X\in \E(J)$ and $Y\in \E(K)$. Then
\begin{align*}
\Hom_{\E(J)}(X,u^\ast Y)&\cong \Hom_{\D(J)}(\iota_J X, \iota_Ju^\ast Y)\\
&\cong\Hom_{\D(J)}(\iota_JX,u^\ast\iota_K Y)\\
&\cong\Hom_{\D(K)}(u_!\iota_J X,\iota_K Y)\\
&\cong\Hom_{\E(K)}(L_Ku_!\iota_J X, Y)
\end{align*}
The first isomorphism is due to  $\iota_J$ being fully faithful, the second because $\iota$ is a morphism of prederivators, the third because $\D$ admits left Kan extensions, and the fourth because $L_K$ is left adjoint to $\iota_K$. Hence we have constructed a left adjoint in $\E$ to $u^\ast$.

For the right Kan extension, we have the following chain of isomorphisms. Using the same notation as above,
\begin{align}\nonumber
\Hom_{\E(J)}(u^\ast Y, X)&\cong\Hom_{\D(J)}(\iota_J u^\ast Y,\iota_J X)\nonumber\\
&\cong \Hom_{\D(J)}(u^\ast\iota_K Y,\iota_J X)\nonumber\\
\label{eq:unitisomorphism}&\cong\Hom_{\D(K)}(\iota_K Y,u_\ast\iota_J X)\\
&\longrightarrow\Hom_{\D(K)}(\iota_K Y,\iota_KL_Ku_\ast\iota_J X)\nonumber\\
&\cong\Hom_{\E(K)}(Y,L_Ku_\ast\iota_J X)\nonumber
\end{align}

We have included here one arrow, because the argument differs here. That map is induced by composition with the unit $\id_{\D(K)}\to \iota_K L_K$ of the adjunction, and we claim that this is an isomorphism in this case. Specifically, we claim that $u_\ast \iota_JX$ is an $L_K$-local object, so that the unit of the adjunction is an isomorphism.

To see this, suppose that $f\colon a\to b$ is an arrow in $\D(K)$ such that $L_K f$ is an isomorphism in $\E(K)$. Then we obtain the following commutative diagram
\begin{equation*}
\vcenter{\xymatrix@C=4em{
\Hom_{\D(K)}(b, u_\ast \iota_J X)\ar[r]^-{-\circ f}\ar[d]_\cong&\Hom_{\D(K)}(a, u_\ast \iota_J X)\ar[d]^\cong\\
\Hom_{\D(J)}(u^\ast b,\iota_J X)\ar[r]\ar[d]_\cong&\Hom_{\D(J)}(u^\ast a,\iota_J X)\ar[d]^\cong\\
\Hom_{\E(J)}(L_Ju^\ast b, X)\ar[r]\ar[d]_\cong&\Hom_{\E(J)}(L_Ju^\ast a, X)\ar[d]^\cong\\
\Hom_{\E(J)}(u^\ast L_K b,X)\ar[r]_-{-\circ u^\ast L_K f}^\cong &\Hom_{\E(J)}(u^\ast L_K a, X)
}}
\end{equation*}

The first two vertical maps are isomorphisms using the adjunctions $(u^\ast,u_\ast)$ and $(L_J,\iota_J)$, and the last vertical isomorphism is because $L$ is a morphism of (pre)derivators. The bottom horizontal arrow is an isomorphism because $L_K f$ is an isomorphism, hence every map in the above diagram is an isomorphism. Therefore $u_\ast\iota_J X$ is an $L_K$-local object.

Therefore the lone arrow in Equation~\ref{eq:unitisomorphism} is also an isomorphism, proving that $L_Ku_\ast\iota_J$ is right adjoint to $u^\ast$ in $\E$. This proves Der3R and Der3L for $\E$.

To show Der4, we need to verify that for any $u\colon J\to K$, $k\in K$, and $X\in \E(J)$, the canonical map
\begin{equation*}
L_{e}\pi_!\iota_{(u/k)}\operatorname{pr}^\ast \!X\to k^\ast L_K u_! \iota_J X
\end{equation*}
is an isomorphism, where $\pi\colon (u/k)\to e$ is the projection to the point and\linebreak $\operatorname{pr}\colon (u/k)\to J$ the forgetful functor. Because $\iota$ and $k$ are morphisms of derivators, we have that $\iota_{(u/k)}\operatorname{pr}^\ast\cong\operatorname{pr}^\ast \iota _J$ and $k^\ast L_K\cong L_e k^\ast$. Hence the above map factors as
\begin{equation*}
\vcenter{\xymatrix{
L_{e}\pi_!\iota_{(u/k)}\operatorname{pr}^\ast X\ar[r]& k^\ast L_K u_! \iota_J X\ar[d]^-\cong\\
L_e\pi_!\operatorname{pr}^\ast \iota_J X\ar[r]\ar[u]^-\cong&L_e k^\ast u_!\iota_J X
}}
\end{equation*}
The bottom map is an isomorphism, as it is $L_e$ applied to the isomorphism guaranteed by Der4 for $\D$ at $\iota_J X\in \D(J)$. Therefore the top map is an isomorphism as well. There is no difference for the right Kan extensions, as the above argument only relied on commuting pullback functors with $L$ and $\iota$. This proves that $\E$ is a derivator.
\end{proof}

In order to determine when such an adjoint as in Lemma~\ref{lemma:localisederivator} might exist, we should compare the derivator situation to the analogous situation in ordinary category theory.
\begin{prop}\label{prop:catloc1}
\cite[Proposition~2.4.1]{Kra10} Let $L\colon\cC\to\cC$ be a functor and \linebreak$\eta\colon\id_\cC\to L$ be a natural transformation. Then the following are equivalent:
\begin{enumerate}
\item $\eta_L\colon L\to L^2$ and $L\eta\colon L\to L^2$ are natural isomorphisms. 
\item There exists a functor $F\colon \cC\to\cD$ with fully faithful right adjoint $G\colon\cD\to\cC$ such that $L=G F$ and $\eta\colon \id_\cC\to G F$ is the unit of the adjunction.
\end{enumerate}
Such a functor $L$ is called a \emph{localization functor}.
\end{prop}

Case (1) is equivalent to the seemingly stronger condition that $L\eta=\eta_L$ and $\eta_L$ is a natural isomorphism. The proof of this fact can be found, for example, at\linebreak\cite[Remark~2.3]{BoyDri06}. Case (2) is the categorical version of Lemma~\ref{lemma:localisederivator}, but we have not given the derivator case of case (1). We need one more proposition to begin that discussion.

\begin{prop}\label{prop:catloc2}
\cite[Proposition~2.6.1]{Kra10} Let $\cC$ be a category and $\cE\subset\cC$ a replete, full subcategory. Then the following are equivalent:
\begin{enumerate}
\item There exists a localization functor $L\colon\cC\to\cC$ with essential image $\cE$.
\item The inclusion functor $\cE\to \cC$ admits a left adjoint.
\end{enumerate}
\end{prop}
In the case that we already have a localization functor, the left adjoint to $\cE\to\cC$ is given by $L$ (with restricted codomain).

We now want to generalize the two propositions above. Suppose that $\E\subset \D$ is a full subprederivator of a derivator $\D$ obtained by taking fully replete subcategories levelwise, and let $G\colon\E\to \D$ be the inclusion. If $G$ admits a left adjoint $F\colon\D\to\E$, then we can define $L=G F\colon\D\to\D$ and $\eta\colon \id_\D\to G F$ the unit of the adjunction. The essential image of $L$ is $\E$; for $Y\in \E(J)$, we have $L(G(Y))=GFG(Y)\cong G(Y)$ as the counit $\varepsilon\colon F G\to \id_\E$ of the $(F,G)$ adjunction is an isomodification.

Finally, using the triangle identities for the $(F,G)$ adjunction we have
\begin{equation}\label{eq:triloc}
L\eta=GF\eta=(G\varepsilon_F)\inv=\eta_{G F}=\eta_L.
\end{equation}
The natural transformation $(G\varepsilon_F)\inv$ only makes sense if $G\varepsilon_F$ is an isomodification, but this is clear as $\varepsilon$ is an isomodification. Therefore $L\eta=\eta_L$ is an isomodification and $L$ is a localization morphism.

Now, suppose we start with an endomorphism $L\colon\D\to\D$ of a derivator $\D$ with a modification $\eta\colon\id_\D\to L$ such that $\eta_L$ and $L\eta$ are isomodifications. Without loss of generality, we may assume $\eta_L=L\eta$, as the proof in the categorical case translates immediately to the derivator context.

Let $\cE_K\subset\D(K)$ denote the essential image of $L_K$ for every $K\in\Dia$. We claim these categories assemble into a prederivator. For $X\in\D(K)$ and $u\colon J\to K$, we have that $u^\ast L_K X\cong~L_J u^\ast X$ because $L$ is a morphism of derivators. If $Y\in \cE_K$, we have $u^\ast Y\cong u^\ast L_K X$ for some $X\in\D(K)$, hence $u^\ast Y\cong L_J u^\ast X$, so $u^\ast Y\in\cE_J$. Because the pullbacks $u^\ast$ restrict to these subcategories, we get a full subprederivator $\E\subset \D$.

Define a morphism $F\colon \D\to \E$ by $F_J(X)=L_J(X)$ and $\gamma_u^F=\gamma_u^L$ and $G\colon \E\to \D$ by the inclusion. Then the unit of the adjunction $\eta\colon \id_\D\to G F$ should be defined by the same $\eta$ as above. For the counit $\varepsilon\colon F G\to \id_\E$, we know that $\eta_Y$ is an isomorphism for any $Y\in\E$, so we let $\varepsilon$ be the inverse of $\eta$. For the triangle identities, we have
\begin{equation*}
F\eta=L\eta=\eta_L=(\varepsilon_L)\inv=(\varepsilon_F)\inv,
\end{equation*}
so $\varepsilon_F F\eta=\id_F$ as required. Similarly, as $G$ is just the inclusion,
\begin{equation*}
\eta_G=G\eta=(G\varepsilon)\inv,
\end{equation*}
so $G\varepsilon\eta_G=\id_G$. We obtain an adjunction, and thus have proved the following:

\begin{prop}\label{prop:locofders}
Let $\D$ be a derivator and $\E\subset\D$ a replete, full subprederivator (\ie levelwise replete, full subcategories). Then the following are equivalent:
\begin{enumerate}
\item There exists a morphism of derivators $L\colon\D\to\D$ with essential image $\E$ and a modification $\eta\colon\id_\D\to L$ such that $\eta_L$ and $L\eta$ are an isomodifications.
\item The levelwise inclusion morphism $\E\to \D$ admits a left adjoint.
\end{enumerate}
\end{prop}

The following general context will be found in the case of the localization of Theorem~\ref{thm:vanishingloc}. Suppose that $F\colon\D\to\D$ is a left adjoint endomorphism of a (pre)derivator~$\D$ such that its right adjoint $G$ is fully faithful on the (essential) image of $F$. Then the assignment $L=GF$ and $\eta\colon\id_\D\to GF$ the unit of the adjunction makes $L$ a localization morphism and the essential image of $F$ is a reflective subprederivator of~$\D$.

\section{Vanishing subderivators}

Not every derivator will be suitable for stabilization. We need two additional properties. We give the first one now and save the second for \S6.

\begin{defn}
A derivator $\D$ is \emph{pointed} if $\D(e)$ is pointed, that is, it contains a zero object.
\end{defn}
Let $K\in\Dia$ and $\pi\colon K\to e$ the projection. If $0\in\D(e)$ is a zero object, the object $\pi^\ast0\in\D(K)$ is easily shown to be a zero object. Therefore every level of a pointed derivator is a pointed category (as the name implies). Moreover, the functors $u_!,u^\ast,u_\ast$ are automatically pointed functors; see \cite[Proposition~3.2]{Gro13}. 

The relatively loose requirement that the base of a derivator admit a zero object implies some more interesting properties. We need a few definitions to describe these properties.

\begin{defn}
Let $u\colon J\to K$ be a fully faithful functor that is injective on objects.
\begin{enumerate}
\item The functor $u$ is a \emph{sieve} if for any morphism $k\to u(j)$ in $K$, $k$ lies in the image of $u$.
\item The functor $u$ is a \emph{cosieve} if for any morphism $u(j)\to k$ in $K$, $k$ lies in the image of $u$.
\end{enumerate}
\end{defn}
\begin{defn}\label{defn:extraadjoint}
\cite[p.6]{Mal07} A derivator $\D$ is \emph{strongly pointed} if for every sieve (resp. cosieve) $u\colon J\to K$ in $\Dia$, $u_\ast$ (resp. $u_!$) admits a right adjoint $u^!$ (resp. left adjoint~$u^?$).
\end{defn}

\begin{prop}\label{prop:stronglypointed}
\cite[Corollaries~3.5~and~3.8]{Gro13} A derivator is pointed if and only if it is strongly pointed.
\end{prop}

There are two further propositions that bear mentioning here.
\begin{prop}\label{prop:fullyfaithfulkanextension}
\cite[Proposition~1.20]{Gro13} Let $\D$ be a derivator and let $u\colon J\to K$ be a fully faithful functor. Then $u_!\colon\D(J)\to\D(K)$ and $u_\ast\colon\D(J)\to\D(K)$ are fully faithful as well.
\end{prop}

\begin{prop}
\cite[Proposition~3.6]{Gro13} Let $\D$ be a pointed derivator.
\begin{enumerate}
\item If $u\colon J\to K$ is a cosieve, then for any $X\in \D(J)$, $k^\ast u_!X= 0$ for any $k\in K\setminus J$ and $j^\ast u_!X\cong j^\ast X$ for any $j\in J$.
\item If $v\colon J\to K$ is a sieve, then for any $Y\in \D(J)$, $k^\ast v_\ast Y= 0$ for any $k\in K\setminus J$ and $j^\ast v_\ast Y\cong j^\ast Y$ for any $j\in J$.
\end{enumerate}
\end{prop}
We colloquially call these `extension by zero' morphisms. For a cosieve $u\colon J\to K$ and $X\in\D(J)$, the underlying diagram of $u_!X$ is isomorphic to $X$ on $J\subset K$ and isomorphic to 0 on the complement. These results are particularly helpful when we attempt to describe with diagrams the action of Kan extensions along sieves and cosieves.

As one final note, suppose that $\D$ is a (strongly) pointed derivator and $u\colon J\to K$ is a sieve. Then the extension by zero morphism $u_\ast\colon\D(J)\to\D(K)$ admits a right adjoint by Proposition~\ref{prop:stronglypointed}, which means that $u_\ast\colon\D^J\to\D^K$ is actually a cocontinuous morphism of derivators. We will use this observation in \S5. We can now begin the discussion of vanishing subderivators in general. We start with a definition.

\begin{defn}
Let $\D$ be a pointed derivator, let $B\in\Dia$, and let $i\colon A\to B$ be a full subcategory. Define $\D(B,A)$ to be the full subcategory of $\D(B)$ given by $\D(B,A)=\{X\in\D(B):i^\ast X=0\in\D(A)\}$.
\end{defn}

\begin{lemma}\label{lemma:vanishingprederivator}
As above, $\D(B,A)$ is the base of a prederivator.
\end{lemma}
\begin{proof}
We let $\D^{B,A}$ denote the purported prederivator. We define $\D^{B,A}(K)$ to be the subcategory of $X\in\D(B\times K)$ such that $(i\times\id_K)^\ast X=0\in\D(A\times K)$. Let $u\colon J\to K$ be any morphism in $\Dia$. Then we have the following commutative square:
\begin{equation}\label{dia:vanishingprederivator}
\vcenter{\xymatrix@C=3em{
A\times J\ar[r]^-{\id_A\times u}\ar[d]_-{i\times\id_J}&A\times K\ar[d]^-{i\times\id_K}\\
B\times J\ar[r]_-{\id_B\times u}&B\times K
}}
\end{equation}
Applying $\D$ to this square, we obtain an equality
\begin{equation*}
(i\times\id_J)^\ast(\id_B\times u)^\ast = (\id_A\times u)^\ast(i\times\id_K)^\ast
\end{equation*}
by strict 2-functoriality. Therefore consider some $X\in\D^{B,A}(K)$. By definition, we have $(i\times\id_K)^\ast X=0$, so $(\id_A\times u)^\ast(i\times\id_K)^\ast X=0$ as well (as restriction functors are pointed). By the above equality, this shows that the object $Y=(\id_B\times u)^\ast X$, which a priori is an element of $\D(B\times J)$ but not necessarily $\D^{B,A}(J)$, satisfies \linebreak$(i\times\id_J)^\ast Y=0$, so it is indeed in the vanishing subcategory.
\end{proof}

The prederivator $\D^{B,A}$ is easy to understand in the case that $i\colon A\to B$ is a sieve (resp. cosieve). The complementary inclusion $j\colon B\setminus A\to A$ is a cosieve (resp. sieve). We can then identify $\D^{B,A}$ with the essential image of $j_!\colon \D^{B\setminus A}\to\D^B$ (resp. $j_\ast$). These morphisms are (levelwise) fully faithful, so $j_!\colon\D^{B\setminus A}\to\D^{B,A}$ is an equivalence of derivators, and thus $\D^{B,A}$ shares all nice properties of $\D$. In the case that $i\colon A\to B$ is not a (co)sieve, the result is identical but the proof significantly harder. The following theorem is \cite[Proposition~7.4]{Hel97}, which was given without proof.

\begin{theorem}\label{thm:vanishingloc}
Let $\D$ be a pointed derivator and $i\colon A\to B$ be a full subcategory in $\Dia$. Then there exists a localization of $\D^B$ with essential image $\D^{B,A}$.
\end{theorem}

Heller claims that the localization morphism is given by the cofiber of the counit of the $(i_!,i^\ast)$ adjunction. We shall see that this is the case, under a sufficiently sophisticated interpretation of the claim. In the case of usual homotopy theory, the construction of cofibers is not functorial. This is one of the main weaknesses of triangulated categories, but one of the strengths of derivators is that we can obtain functorial -- and even coherent, as we will shortly show -- cofibers.

\begin{proof}
Our overall goal is the following: there is an adjunction $(L,R)\colon\D^B\to\D^B$ such that $\D^{B,A}$ is the essential image of $L$ and $R$ is fully faithful on $\D^{B,A}$. As we noted after Proposition~\ref{prop:locofders}, this gives us a localization $RL\colon\D^B\to\D^B$.

\begin{cons}\label{cons:coherentcounit}
For any functor $u\colon J\to K$, there is a morphism of derivators\linebreak $\D^J\to\D^{J\times [1]}$ given by
\begin{equation*}
X\mapsto (u_!u^\ast X\to X).
\end{equation*}
That is, the counit of the $(u_!,u^\ast)$ adjunction may be constructed coherently.
\end{cons}
The idea for this construction originates from personal correspondence with Kevin Carlson and the proof is joint with Ioannis Lagkas.

Consider the category $\Cyl(u)$ constructed as follows: its objects are the disjoint union of $J$ and $K$, and its morphisms are defined as follows:
\begin{equation*}
\Hom_{\Cyl(u)}(x,y)=\begin{cases}\Hom_J(x,y)&x,y\in J\\\Hom_K(x,y)&x,y\in K\\\ast&x\in J,y\in K\text{ and }u(x)=y\\\varnothing&\text{otherwise}\end{cases}
\end{equation*}
In other words, the category is $J$ glued to $K$ along the map $u$, so we will refer to this as a mapping cylinder. We will abuse notation and refer to $J,K\subset\Cyl(u)$ as full subcategories. Let $\overline u\colon \Cyl(u)\to K\times[1]$ be defined by
\begin{equation*}
\overline u(x)=\begin{cases} (u(x),0)&x\in J\\ (x,1)&x\in K\end{cases}
\end{equation*}
with the action of $\overline u$ on the maps in $J,K$ obvious and sending the unique `gluing' map $j\to k$ in $\Cyl(u)$ to the vertical map $(u(j),0)\to(k,1)$ as $u(j)=k$. Let $p\colon K\times[1]\to K$ be the projection.

We claim that $\overline {u}_!\overline {u}^\ast p^\ast\colon \D^K\to\D^{K\times[1]}$ gives coherently the counit of the $(u_!,u^\ast)$ adjunction. To verify this, we use the following diagram:
\begin{equation}\label{dia:mates}
\vcenter{\xymatrix@C=3em{
J\ar[r]^-{i_0}\ar[d]_-{u}&\Cyl(u)\ar[r]^-{q}\ar[d]^-{\overline u}\ar@{}[dl]|\swtrans\ar@{}[dl]<-1.2ex>|\alpha&K\ar[dd]^-{\id_K}\ar@{}[ddl]|\swtrans\ar@{}[ddl]<-1.2ex>|\gamma\\
K\ar[r]^-s\ar[d]_-{\id_K}&K\times[1]\ar@{}[dl]|\swtrans\ar@{}[dl]<-1.2ex>|\beta\ar[d]^-{\id_{K\times[1]}}\\
K\ar[r]_-t&K\times[1]\ar[r]_-p&K
}}
\end{equation}
In the above, we write $q=p\overline u$; $s,t\colon K\to K\times[1]$ are the inclusion into the source $K\times\{0\}$ and target $K\times\{1\}$ of the coherent arrow, respectively; and $i_0\colon J\to \Cyl(u)$ is the inclusion of the full subcategory. The transformation $\alpha$ is the identity, $\beta$ is the transformation $s\Rightarrow t$ as in the underlying diagram functor $\dia_{[1],K}$, and $\gamma$ is also the identity.

We will use more heavily here the machinery of the calculus of mates, and so add more details to the construction addressed in Der4. Following the notation of \cite[Definition~1.15]{Gro13}, given a commutative square in $\Dia$, we get a commutative square in $\CAT$ after applying a derivator $\D$:
\begin{equation}\label{dia:exactsquare}
\vcenter{\xymatrix{
J_1\ar[r]^-v\ar[d]_-{u_1}&J_2\ar[d]^-{u_2}\ar@{}[dl]|\swtrans\ar@{}[dl]<-1.2ex>|\delta\\
K_1\ar[r]_-{w}&K_2
}}\mapsto
\vcenter{\xymatrix{
\D(J_1)&\D(J_2)\ar[l]_-{v^\ast}\ar@{}[dl]|\swtrans\ar@{}[dl]<-1.2ex>|{\delta^\ast}\\
\D(K_1)\ar[u]^-{u_1^\ast}&\D(K_2)\ar[l]^-{w^\ast}\ar[u]_-{u_2^\ast}
}}
\end{equation}
By Der3, the functors $u_1^\ast,u_2^\ast$ admit left adjoints, which along with the counit of $(u_{1,!},u_1^\ast)$ and the unit of $(u_{2,!},u_2^\ast)$ gives us an augmented diagram:
\begin{equation*}
\vcenter{\xymatrix{
\D(K_1)&\D(J_1)\ar[l]_-{u_{1,!}}\ar@{}[dl]|\swtrans&\D(J_2)\ar[l]_-{v^\ast}\ar@{}[dl]|\swtrans\ar@{}[dl]<-1.2ex>|{\delta^\ast}&\ar@{}[dl]|\swtrans\\
&\D(K_1)\ar@(l,d)[ul]^-{=}\ar[u]^-{u_1^\ast}&\D(K_2)\ar[l]^-{w^\ast}\ar[u]_-{u_2^\ast}&\D(J_2)\ar[l]^-{u_{2,!}}\ar@(u,r)[lu]_-{=}
}}
\end{equation*}
In total, we obtain a transformation $\delta_!\colon u_{2,!}v^\ast\to w^\ast u_{1,!}$ which we call the \emph{left mate} of $\delta$. Taking mates is compatible with pasting as well; see \cite[Lemma~1.14]{Gro13}.

The transformation we must study is $\beta^\ast=\D(\beta)$ whiskered with some functors, namely
\begin{equation*}
\beta_{!,\overline i_!q^\ast}\colon \id_{B,!}s^\ast\overline i_!q^\ast\to t^\ast \id_{B\times[1],!}\overline i_!q^\ast,
\end{equation*}
which gives us the coherent map represented by $\overline i_!\overline i^\ast p^\ast X$ for $X\in\D^B$. We provide the functors $\id_!$ to make more clear the calculus of mates, though they are isomorphisms.

The map we must study is, for $X\in\D^K$, 
\begin{equation*}
\dia_{[1],K} (\overline u_! \overline u^\ast p^\ast X)\colon s^\ast \overline {u}_! \overline {u}^\ast p^\ast X\to t^\ast \overline {u}_! \overline {u}^\ast p^\ast X
\end{equation*}
which is encoded by the natural tranformation $\beta$. To be more specific, the left mate of $\beta^\ast$ is the natural transformation
\begin{equation*}
\beta_!\colon \id_{K,!}s^\ast\to t^\ast \id_{K\times[1],!}.
\end{equation*}
Since the left Kan extension along an identity is an isomorphism, in essence we have a transformation $\beta_!\colon s^\ast\Rightarrow t^\ast$. We then need to whisker this transformation with $\overline{u}_! q^\ast=\overline{u}_!\overline{u}^\ast p^\ast$ to obtain the desired map. We pick the left mate (rather than using $\beta^\ast$ itself) because we need to whisker with a left Kan extension.

We would like to understand how the transformation $\beta_{!,\overline u_! q^\ast}$ fits into the left mate of total pasting of the diagram. To that end, we take the left mates of $\alpha,\beta,\gamma$ and paste them as follows:
\begin{equation*}
\vcenter{\xymatrix@C=3em{
\D(J)\ar@{<-}[r]^-{i_0^\ast}\ar[d]_-{u_!}&\D(\Cyl(u))\ar@{<-}[r]^-{q^\ast}\ar[d]^-{\overline u_!}\ar@{}[dl]|\swtrans\ar@{}[dl]<-1.2ex>|{\alpha_!}&\D(K)\ar[dd]^-{\id_{K,!}}\ar@{}[ddl]|\swtrans\ar@{}[ddl]<-1.2ex>|{\gamma_!}\\
\D(K)\ar@{<-}[r]^-{s^\ast}\ar[d]_-{\id_{K,!}}&\D(K\times[1])\ar@{}[dl]|\swtrans\ar@{}[dl]<-1.2ex>|{\beta_!}\ar[d]^-{\id_{K\times[1],!}}\\
\D(K)\ar@{<-}[r]_-{t^\ast}&\D(K\times[1])\ar@{<-}[r]_-{p^\ast}&\D(K)
}}
\end{equation*}
Working through these left mates one at a time, we obtain the composite transformation
\begin{equation*}
\vcenter{\xymatrix@C=4em{
\id_{K,!}u_!i_0^\ast q^\ast\ar[r]^-{\id_{K,!}\alpha_{!,p^\ast}}&\id_{K,!}s^\ast\overline u_! p^\ast \ar[r]^-{\beta_{!,\overline u_!q^\ast}}& t^\ast \id_{K\times[1],!}\overline u_!q^\ast\ar[r]^-{t^\ast\gamma_!}&t^\ast p^\ast \id_{K,!}\!.
}}
\end{equation*}

The total pasting of the diagram is
\begin{equation*}
\vcenter{\xymatrix{
J\ar[r]^-u\ar[d]_-u&K\ar[d]^-{\id_K}\ar@{}[dl]|\swtrans\ar@{}[dl]<-1.25ex>|\id\\
K\ar[r]_-{\id_K}&K
}}
\end{equation*}
and its left mate is $\varepsilon\colon u_!u^\ast\to\id_K^\ast\id_{K,!}\cong\id_{\D^K}$, \ie the counit of the $(u_!,u^\ast)$ adjunction. Thus the above composite is the counit of the adjunction by\linebreak \cite[Lemma~1.14(1)]{Gro13}. Therefore it suffices to show that $\alpha_{!,p^\ast}$ and $t^\ast \gamma_!$ are natural isomorphisms, thereby obtaining
\begin{equation}\label{dia:mategoal}
\vcenter{\xymatrix@C=2em{
\id_{K,!}i_!i_0^\ast q^\ast\ar[d]_-{\id_{K,!}\alpha_{!,p^\ast}}^-\cong\ar[r]^-\cong&u_!u^\ast\ar[r]^-\varepsilon&\id_{\D^K}\ar[r]^-\cong&t^\ast p^\ast \id_{K,!}\!\\
\id_{K,!}s^\ast\overline u_! p^\ast \ar[rrr]_-{\beta_{!,\overline u_!q^\ast}}&&& t^\ast \id_{K\times[1],!}\overline u_!q^\ast\ar[u]_-{t^\ast\gamma_!}\ar[u]^-\cong
}}
\end{equation}
so that $\beta_{!,\overline u_!q^\ast}$ is (up to isomorphism) the counit of the adjunction.

We begin with $\alpha_!$. We claim that $\alpha_!$ is actually a natural isomorphism even without whiskering with $p^\ast$. We use~\cite[Proposition~1.24]{Gro13}: if that square is a (1-categorical) pullback and the bottom horizontal functor is a Grothendieck fibration or the right vertical functor is a Grothendieck opfibration, then $\alpha_!$ is an isomorphism. For us, the square is a pullback, which requires the transformation $\alpha$ to be the identity (and so it is). Moreover, $s$ is a sieve, which in particular is a discrete Grothendieck fibration. This proves that $\alpha_!$, hence $\alpha_{!,p^\ast}$, is an isomorphism.

For $\gamma_!$, it is not true that this transformation is an isomorphism in general, but it is after applying $t^\ast$. To prove this, we give another pasting
\begin{equation}\label{dia:pasting2}
\vcenter{\xymatrix{
K\ar[r]^-{i_1}\ar[d]_-{\id_K}&\Cyl(u)\ar[r]^-q\ar[d]_-{\overline u}\ar@{}[dl]|\swtrans\ar@{}[dl]<-1.25ex>|\theta&K\ar[d]^-{\id_K}\ar@{}[dl]|\swtrans\ar@{}[dl]<-1.25ex>|\gamma\\
K\ar[r]_-t&K\times[1]\ar[r]_-p&K
}}
\end{equation}
where $i_1\colon K\to \Cyl(u)$ is the inclusion into the bottom of the mapping cylinder and $\theta$ is just the identity. Taking left mates and composing, we obtain
\begin{equation*}
\vcenter{\xymatrix{
\id_{K,!}i_1^\ast q^\ast\ar[r]^-{\theta_{!,q^\ast}}&t^\ast \overline u_! q^\ast\ar[r]^-{t^\ast\gamma_!}&t^\ast p^\ast \id_{K,!}\!.
}}
\end{equation*}
But the total pasting is just
\begin{equation*}
\vcenter{\xymatrix{
B\ar[r]^-{\id_B}\ar[d]_-{\id_B}&B\ar[d]^-{\id_B}\ar@{}[dl]|\swtrans\ar@{}[dl]<-1.25ex>|\id\\
B\ar[r]_-{\id_B}&B
}}
\end{equation*}
whose left mate is clearly an isomorphism. Therefore we just need to show that $\theta_{!,q^\ast}$ is an isomorphism to complete this argument.

The functor $\overline u$ is a cosieve, and a fortiori a discrete Grothendieck opfibration. Moreover, the lefthand square in Diagram~\ref{dia:pasting2} is clearly a pullback square, and has the identity natural transformation, hence by~\cite[Proposition~1.24]{Gro13} again the transformation $\theta_!$ is an isomorphism, even without whiskering with $q^\ast$. We thus obtain
\begin{equation*}
\vcenter{\xymatrix{
\id_{K,!}i_1^\ast q^\ast\ar[r]^-{\theta_{!,q^\ast}}_-\cong\ar@(d,d)[rr]_-\cong&t^\ast \overline u_! q^\ast\ar[r]^-{t^\ast\gamma_!}&t^\ast p^\ast \id_{K,!},
}}
\end{equation*}
proving that $t^\ast\gamma_!$ is an isomorphism. This proves that Diagram~\ref{dia:mategoal} has isomorphisms where claimed and so the functor $\overline {u}_!\overline {u}^\ast p^\ast\colon\D^K\to\D^{K\times[1]}$ is a coherent counit morphism
\begin{equation*}
X\mapsto (u_!u^\ast X\to X).
\end{equation*}

We will use this construction for the inclusion $i\colon A\to B$. Having constructed the counit coherently, we now take the cone of the counit via a canonical morphism $C\colon \D^{[1]}\to \D$. In brief, for $(f\colon x\to y)$ in $\D^{[1]}$, $C(f)$ is computed as the pushout along zero
\begin{equation*}
\vcenter{\xymatrix{
x\ar[r]^-f\ar[d]&y\ar[d]\\
0\ar[r]&C(f)
}}
\end{equation*}
We will give the details of this construction for the second half of this proof.

In total this gives us a morphism
\begin{equation*}
\xymatrix@C=3em{
L\colon\D^{B}\ar[r]^-{\overline {i}_!\overline {i}^\ast p^\ast}&\D^{B\times [1]}\ar[r]^-{C} & \D^B.
}
\end{equation*}

We now need to show that the image of $L$ is contained in $\D^{B,A}$. To that end, we consider $Z=i^\ast L(X)\in \D^A$ for any $X\in\D^B$. Because $i^\ast$ is cocontinuous, it commutes with $C$, so we can compute $Z$ (up to canonical isomorphism) as the pushout
\begin{equation*}
\xymatrix{
i^\ast i_!i^\ast X\ar[r]^-{i^\ast\varepsilon_X}\ar[d]&i^\ast X\ar[d]\\
0\ar[r]& Z
}
\end{equation*}
We claim this top map is an isomorphism. Because $i_!\colon\D^A\to\D^B$ is fully faithful by Proposition~\ref{prop:fullyfaithfulkanextension}, the unit $\eta\colon \id_{\D^A}\to i^\ast i_!$ is an isomorphism, and one of the triangle identities gives us
\begin{equation*}
\xymatrix{
i^\ast X\ar[r]^-{\eta_{i^\ast \!X}}_-\cong\ar@(d,l)[dr]_= & i^\ast i_! i^\ast X\ar[d]^{i^\ast \varepsilon_{X}}\\
& i^\ast X
}
\end{equation*}
so $i^\ast\varepsilon_X$ is also an isomorphism by two out of three. By \cite[Proposition~3.12]{Gro13}, the pushout of an isomorphism is an isomorphism, hence $Z=0$. Therefore $L(X)\in\D^{B,A}$ by definition. 

Suppose that $Y\in\D^{B,A}$. Then to compute $L(Y)$, we have the pushout
\begin{equation*}
\vcenter{\xymatrix{
i_!i^\ast Y\ar[r]\ar[d]&Y\ar[d]\\
0\ar[r]&L(Y)
}}
\end{equation*}
Because $Y\in\D^{B,A}$, $i^\ast Y=0$, so the lefthand vertical map is an isomorphism. Hence the map $Y\to L(Y)$ is also an isomorphism. Therefore the essential image of $L$ is precisely $\D^{B,A}$.

Every part of the construction of $L$ is a left adjoint morphism of derivators, so $L$ has a right adjoint $R\colon\D^B\to\D^B$. We now prove that $R$ is fully faithful on $\D^{B,A}\subset \D^B$. To that end, again let $Y\in\D^{B,A}$.

The cone morphism $C\colon \D^{[1]}\to \D$ is a composition of three morphisms. First, let $K$ denote the category
\begin{equation*}
\vcenter{\xymatrix@C=1em@R=1em{
(0,0)\ar[r]\ar[d]& (1,0)\ar[d]\\
(0,1)\ar[r]&(1,1)
}}
\end{equation*}
Further, let $J$ denote the full subcategory of $K$ without the element $(1,1)$. Let $i_J\colon J\to K$ be the natural inclusion and let $i_{[1]}\colon [1]\to J$ be the inclusion of the horizontal arrow.

The first step of $C$ is the extension by zero morphism $i_{[1],\ast}\colon \D^{[1]}\to\D^J$. The second step is the left Kan extension $i_{J,!}\colon\D^J\to\D^K$ which computes the pushout. The final step is the restriction $(1,1)^\ast\colon\D^K\to\D$ to the terminal object of $K$, which is the cone of the coherent morphism we began with.

The right adjoint to $(1,1)^\ast$ is $(1,1)_\ast$. By \cite[Lemma~1.19]{Gro13}, $(1,1)_\ast$ is isomorphic to the constant diagram functor $\pi_K^\ast$, so that
\begin{equation*}
(1,1)_\ast Y\,\cong \vcenter{\xymatrix@R=1em@C=1em{
Y\ar[r]\ar[d]&Y\ar[d]\\
Y\ar[r]&Y
}}
\end{equation*}
with all morphisms identities. The right adjoint to $i_{J,!}$ is $i_J^\ast$, which restricts the constant diagram to the upper-left corner.  The right adjoint to $i_{[1],\ast}$ is $i_{[1]}^!$, an exceptional right adjoint. The construction of this adjoint as a composition of three morphisms can be found following \cite[Corollary~3.8]{Gro13}. The first step is the extension by zero
\begin{equation*}
\vcenter{\xymatrix@R=1em@C=1em{
&Y\ar[r]\ar[dl]&Y\\
Y
}}\mapsto
\vcenter{\xymatrix@R=1em@C=1em{
0\ar[dd]\\
&Y\ar[r]\ar[dl]&Y\\
Y
}}
\end{equation*}

Call the resulting diagram shape $J'$ and the resulting object $Y'\in\D^{B\times J'}$. The second step is to compute the right Kan extension along the inclusion $s\colon J'\to J\times[1]$. The last step will be to restrict ourselves to $(0,0,0)\to (1,0,0)$ in $J\times [1]$. It will suffice to compute $(0,0,0)^\ast s_\ast Y'$ and $(1,0,0)^\ast s_\ast Y'$, which we can do using Der4. We have that
\begin{equation*}
(1,0,0)^\ast s_\ast Y'\cong \pi_{J'_{(1,0,0)/},\ast}\operatorname{pr}^\ast Y',
\end{equation*}
where $\operatorname{pr}\colon J'_{(1,0,0)/}\to J'$ is the projection from the slice category and\linebreak $\pi_{J'_{(1,0,0)/}}\colon J'_{(1,0,0)/}\to e$ is the projection. Examining the slice category, we note that $(1,0,0)$ maps only to $(1,0,1)\in J\times [1]$. Therefore the slice category consists only of the element $(1,0,1)$ and the unique map $(1,0,0)\to(1,0,1)$, so that $\pi_{J'_{(1,0,0)/}}\cong \id_e$ and
\begin{equation*}
 \pi_{J'_{(1,0,0)/},\ast}\operatorname{pr}^\ast Y'\cong \id_\ast (1,0,1)^\ast Y'\cong Y.
\end{equation*}

The object at $(0,0,0)$ is a little more complicated. The slice category $J'_{(0,0,0)/}$ is actually all of $J'$ because $(0,0,0)$ is the initial element of $J\times[1]$. Therefore
\begin{equation*}
(0,0,0)^\ast s_\ast Y'\cong \pi_{J',\ast} Y'.
\end{equation*}
Fortunately, $J'$ contains a homotopy initial subcategory. The restriction along any right adjoint functor is homotopy final, see \cite[Proposition~1.18]{Gro13}, so dually we have that if $\ell\colon I\to J'$ is a left adjoint functor,
\begin{equation*}
\pi_{J',\ast}Y'\cong \pi_{I,\ast}\ell^\ast Y'.
\end{equation*}

Let $\ell\colon I\to J'$ be the inclusion of the full subcategory excluding the element $(1,0,1)$. The right adjoint $r\colon J'\to I$ is the identity on $I$ but takes $(1,0,1)\mapsto (0,0,1)$. Because $(1,0,1)$ is only the codomain of one map in $J'$, we just need to verify that
\begin{equation*}
\Hom_{J'}(\ell(0,0,1),(1,0,1))\cong\Hom_{I}((0,0,1),r(1,0,1)).
\end{equation*}
But each of these are one-point sets, so we indeed have an adjunction. Naturality of the hom-set bijection follows because we are working in posets. Therefore we need to compute the limit of 
\begin{equation*}
\ell^\ast Y'\,=\vcenter{\xymatrix@R=1em@C=1em{
0\ar[dd]\\
&Y\ar[dl]\\
Y
}}
\end{equation*}
which is just the pullback of that diagram. Because the map $Y\to Y$ is an isomorphism and the pullback of an isomorphism is still an isomorphism, we have $\pi_{I,\ast}\ell^\ast Y'\cong 0$.

Therefore the underlying diagram of $i_{[1]}^\ast s_\ast Y'$ is $(0\to Y)\in\D^{B\times [1]}$. The underlying morphism must be the zero morphism, so we do not need to identify it using the calculus of mates as we did in Construction~\ref{cons:coherentcounit}. The remaining right adjoints are the restriction along $\overline i\colon \overline A\to B\times [1]$ followed by the right Kan extension along the composition $q\colon\overline A\to B\times [1]\to B$. First, the object $\overline i^\ast (0\to Y)$ we could write as $(0_A\to Y)$, so that we have forgotten the elements in $B\setminus A$ in the domain of the coherent arrow.

Now we finally use the assumption that $Y\in\D^{B,A}$. We know that $i^\ast Y=0_A$, so $(0_A\to Y)={q}^\ast Y\in \D^{\overline A}$. Therefore the right adjoint to $L$ evaluated at $Y\in\D^{B,A}$ is exactly $q_\ast {q}^\ast Y$. We now claim that the unit $\id_{\D^B}\to q_\ast {q}^\ast$ is an isomorphism, so that the right adjoint to $L$ is isomorphic to the identity on elements of $\D^{B,A}$.

By usual category theory, the unit of $({q}^\ast,q_\ast)$ is an isomorphism if and only if ${q}^\ast$ is fully faithful. We now show that this is the case. $q\colon \overline A\to B$ itself is left adjoint to the functor $i_1\colon B\to \overline A$ defined by $b\mapsto (b,1)$. To illustrate this, let $(a,n)\in\overline A$ and $b\in B$. Then
\begin{equation*}
\Hom_{B}(q(a,n),b)=\Hom_B(a,b)\cong \Hom_{\overline A}((a,n),(b,1))=\Hom_{\overline A}((a,n),i_1(b))
\end{equation*}
The adjunction $(q,i_1)$ in $\Dia$ under $\D$ becomes an adjunction $(i_1^\ast,{q}^\ast)$ from $\D^B$ to $\D^{\overline A}$. Because adjoints are essentially unique, we have that ${q}^\ast\cong i_{1,\ast}$. Because $i_1$ is fully faithful, $i_{1,\ast}$ is also fully faithful by Proposition~\ref{prop:fullyfaithfulkanextension}, which proves ${q}^\ast$ is as well.

Hence we finally may conclude that the morphism of prederivators $L\colon \D^B\to \D^B$ has a right adjoint which is fully faithful on the essential image of $L$. We have thus accomplished the plan at the beginning of this proof, and so obtain a localization of the derivator $\D^B$ with essential image $\D^{B,A}$.
\end{proof}

\begin{cor}
The prederivator $\D^{B,A}$ is a pointed derivator.
\end{cor}
\begin{proof}
Because $\D^{B,A}$ is a localization of a derivator, Lemma~\ref{lemma:localisederivator} applies, so $\D^{B,A}$ is a derivator. Clearly $0\in \D(B)$ is in the subcategory $\D(B,A)$, so $\D^{B,A}$ is pointed as well.
\end{proof}

\begin{remark}\label{rk:localizekanformula}
The left and right Kan extensions of the localized derivator $\D^{B,A}$ have explicit formulas based on those of $\D^B$. For a functor $u\colon J\to K$, write $u'_!$ and $u'_\ast$ for the left and right Kan extensions of $u$ in $\D^{B,A}$. Then the proof of Lemma~\ref{lemma:localisederivator} provides
\begin{equation*}
u'_!=L_Ku_!R_J\quad\text{and}\quad u'_\ast=L_Ku_\ast R_J.
\end{equation*}

However, the restriction of $u_!$ and $u_\ast$ to $\D^{B,A}$ have their images in $\D^{B,A}$ as well. Because the morphism of derivators $i^\ast\colon\D^B\to\D^A$ is continuous and cocontinuous and $u_\ast$ is pointed, we have for any $X\in\D^{B}(J)$,
\begin{equation*}
i^\ast u_\ast X\cong u_\ast i^\ast X\cong u_\ast 0\cong 0.
\end{equation*}
Therefore $u_\ast X\in \D^{B,A}(K)$ already. Similarly, $u_!X\in \D^{B,A}(K)$. This is a much quicker way to prove the above corollary, but does not give us the localization morphism which will prove instrumental in the next section.

The situation of vanishing subderivators is exceptional among localizations of derivators. In general, the localized derivator $\E\subset\D$ need not be closed under both left and right Kan extensions in $\D$, but will obtain its own left and right Kan extensions as above.
\end{remark}

\begin{remark}
There is also a construction of $\D^{B,A}$ as a colocalization of $\D^B$, \ie the inclusion $\D^{B,A}\to\D^B$ admits a right adjoint. This is the case if and only if the dual situation of Proposition~\ref{prop:locofders}~(1) holds. We obtain the colocalization morphism $\D^B\to\D^B$ by composing an adjunction similar to the one above.

We can construct the coherent unit of the $(i^\ast,i_\ast)$ adjunction and take its fiber, which defines $R'\colon\D^B\to\D^B$. We can see that the image of $R'$ is contained in $\D^{B,A}$ because $Z'=i^\ast R'(X)$ for $X\in\D^B$ is computed as the pullback
\begin{equation*}
\xymatrix{
Z'\ar[r]\ar[d]&i^\ast X\ar[d]^-{i^\ast\eta_X}\\
0\ar[r]&i^\ast i_\ast i^\ast X
}
\end{equation*}
where the righthand map is an isomorphism for reasons dual to the above. Hence $Z'=0$ so $R'(X)\in\D^{B,A}$. The rest of the proof proceeds similarly.
\end{remark}

\section{Prespectra in derivators}

We now need to define what we mean by stable in the theory of derivators. Any derivator admits an intrinsic notion of suspension and loop endofunctors.
\begin{notn}\label{notn:square}
Let $\square$ be the category
\begin{equation*}
\xymatrix@C=1em@R=1em{
(0,1)\ar[r]&(1,1)\\
(0,0)\ar[r]\ar[u]&(1,0)\ar[u]
}
\end{equation*}
Let $i_{\ll}\colon\ll\to\square$ be the full subcategory lacking the element $(1,1)$ and $i_\ur\colon\ur\to\square$ the full subcategory lacking $(0,0)$.
\end{notn}

\begin{remark}
In most papers on derivators, the category $\square$ is oriented so that the top-left object is initial and the bottom-right object is final. However, as the majority of diagrams in this paper are subcategories of $\N^2$, we have oriented $\square$ using the usual Cartesian coordinates.
\end{remark}

\begin{defn}\label{defn:suspensionloop}
Let $\D$ be a derivator. Define the \emph{suspension} endomorphism \linebreak$\Sigma\colon\D\to\D$ by the composition
\begin{equation*}
\vcenter{\xymatrix@C=3em{
\D\ar[r]^-{(0,0)_\ast}&\D^\ll\ar[r]^-{i_{\ll,!}}&\D^\square\ar[r]^-{i_\ur^\ast}&\D^\ur\ar[r]^-{(1,1)^\ast}&\D.
}}
\end{equation*}
Define the \emph{loop} endomorphism $\Omega\colon\D\to\D$ by the composition
\begin{equation*}
\vcenter{\xymatrix@C=3em{
\D\ar[r]^-{(1,1)_!}&\D^\ur\ar[r]^-{i_{\ur,\ast}}&\D^\square\ar[r]^-{i_\ll^\ast}&\D^\ll\ar[r]^-{(0,0)^\ast}&\D.
}}
\end{equation*}
\end{defn}
We write the restriction in two steps to emphasize the following proposition.
\begin{prop}
\cite[Proposition~3.17]{Gro13} Let $\D$ be a pointed derivator. Then $(\Sigma,\Omega)$ is an adjunction of endomorphisms of $\D$.
\end{prop}

\begin{defn}
A pointed derivator $\D$ is \emph{stable} if $(\Sigma,\Omega)$ is an adjoint equivalence of derivators.
\end{defn}

\begin{remark}
In the past, the descriptor stable referred to \emph{strong} derivators in which $(\Sigma,\Omega)$ is an adjoint equivalence, giving a canonical triangulated structure on each $\D(K)$. Lately it has become more common to call such a derivator \emph{triangulated} instead, especially given the existence of stable derivators which cannot admit a canonical triangulation; see \cite[Remark~5.4]{Lag17}.
\end{remark}

There are other perspectives on stability in derivators. Groth in \cite{Gro16b} has explored some equivalent definitions of stability, and the interested reader is directed there. Our aim for stabilization is to force the $(\Sigma,\Omega)$ adjunction to be an equivalence. We cannot do this solely in $\D(e)$, but will use the higher structure of the derivator $\D$ in a way recalling \cite[Definition~1.1]{Hov01}.

\begin{defn}\label{defn:prespectrum}
Consider the poset $\Z^2$ viewed as a category. Let $V\subset\Z^2$ be the full subcategory of those $(i,j)$ such that $|i-j|\leq 1$. Let $\partial V$ be the full subcategory of $V$ on $(i,j)$ such that $|i-j|=1$. We define the \emph{pointed derivator of prespectrum objects on $\D$} to be $\Sp\D:=\D^{V,\partial V}$.
\end{defn}

\begin{notn}
In the introduction, $\Sp$ referred to the actual spectrum objects associated to an $\infty$-category or model category. We follow Heller's notation in using $\Sp$ for the prespectrum objects and will use $\St$ to refer to (stable) spectrum objects.
\end{notn}

We will often write $X_i$ for $(i,i)^\ast X$ for brevity. This should cause no confusion as $(i,j)^\ast X=0$ if $i\neq j$ by the vanishing criterion. A prespectrum $X\in\Sp\D$ then has underlying diagram
\begin{equation*}
\xymatrix@C=1em@R=1em{
&{}&{}&{}&{}\\
&&0\ar[r]&X_1\ar@.[ur]&{}\\
&0\ar[r]&X_0\ar[r]\ar[u]&0\ar[u]\\
&X_{-1}\ar[r]\ar[u]&0\ar[u]\\
\ar@.[ur]&
}
\end{equation*}
extending infinitely in both directions.

To justify the definition, we need to describe in what sense a prespectrum contains information about the $(\Sigma,\Omega)$ adjunction on $\D$. On the surface, a prespectrum $X\in~\Sp\D$ is equivalent to a discrete collection of objects $\{X_i\}$ for $i\in\Z$. However, the category $V$ and the vanishing of $\partial V$ naturally encode the structure maps of the prespectrum.

Let $X\in \Sp\D$, and let $i_n\colon\square\to V$ be the functor defined by $(a,b)\mapsto (a+n,b+n)$. Then we have
\begin{equation*}
i_n^\ast X\,=
\vcenter{\xymatrix@R=1em@C=1em{
0\ar[r]&X_{n+1}\\
X_n\ar[u]\ar[r]&0\ar[u]
}}\,\in\D^{\square}.
\end{equation*}

We follow \cite[\S2]{Gro16a} for the following notation. The category $\square$ is the cone on the category $\ur$, that is, formally adding an initial object to $\ur$. Recall from Definition~\ref{defn:suspensionloop} that the loop of an object $Y\in\D$ is $(0,0)^\ast i_{\ur,\ast}$ applied to the following diagram:
\begin{equation*}
\vcenter{\xymatrix@R=1em@C=1em{
0\ar[r]&Y\\
&0\ar[u]
}}
\end{equation*}
From the square $i_n^\ast X$, we would like to ignore the corner at $(0,0)$ and apply $i_{\ur,\ast}$ to the rest of the diagram. This should give us the object $\Omega X_{n+1}$, and we can now describe how close $X_n$ is to $\Omega X_{n+1}$. This will be the $n$th structure map of $X$.

This occurs formally as follows: let $A$ be the cone on the category $\square$, \ie the double cone on $\ur$. Denote by $\varnothing$ the object that is added to $\square$ to form $A$. Let $s_\square\colon \square\to A$ be the functor sending $(0,0)$ to $\varnothing$ and which is the identity elsewhere. Then it is a straightforward computation by Der4 that
\begin{equation}\label{dia:comparisonmorphism}
s_{\square,\ast} i_n^\ast X\,=
\vcenter{\xymatrix@R=1em@C=1em{
&0\ar[r]&X_{n+1}\\
&\Omega X_{n+1}\ar[r]\ar[u]&0\ar[u]\\
X_n\ar[ur]\ar@(u,l)[uur]\ar@(r,d)[urr]
}}\,\in\D^A.
\end{equation}
If we restrict to the map $\varnothing\to (0,0)$ in $A$, this gives us a coherent map $(X_n\to \Omega X_{n+1})$ in $\D^{[1]}$. We could have gotten \emph{incoherent} structure maps using the universal property of the pullback in $\D(e)$, but we require coherence for what follows.

Repeating this construction for all $n\in\Z$ gives us the structure maps of $X\in~\Sp\D$. We will make a formal definition later, but we can anticipate that a (stable) spectrum is a prespectrum such that all of these structure maps are isomorphisms.

We will first describe how to pass from $\D$ to $\Sp\D$ in a canonical way, then in \S6 discuss stable spectra and the stabilization morphism.

\begin{prop}
The morphism $(0,0)^\ast\colon \Sp\D\to\D$ sending $X\mapsto X_0$ admits a left adjoint.
\end{prop}
\begin{remark}
For any prederivator $\E$, we can define morphisms $\Sp\D\to\E$ by using the composition $\Sp\D\subset\D^V\to\E$, but we cannot necessarily define morphisms $\E\to\Sp\D$ in the same way. Without more information there is no reason to suspect that an arbitrary morphism $\E\to\D^V$ will vanish on $\partial V$.

We know that the morphism $(0,0)_!$ is left adjoint to $(0,0)^\ast\colon\D^V\to \D$, but $(0,0)_!\colon\D\to\D^V$ is not the adjoint we are looking for. More specifically, let us calculate $(1,0)^\ast(0,0)_! x$ for some $x\in \D$. By Der4, we have a canonical isomorphism 
\begin{equation*}
\pi_{e_{/(1,0)},!}\operatorname{pr}^\ast\!x\to (1,0)^\ast(0,0)_!x,
\end{equation*}
where $\operatorname{pr}\colon e_{/(1,0)}\to e$ is the canonical projection from the slice category and $\pi_{e_{/(1,0)}}$ is the unique map to the terminal category $e$. But since $V$ is a poset, the category $e_{/(1,0)}$ contains only one element, namely $(e,(0,0)\to (1,0))$. Therefore $\operatorname{pr}$ and $\pi_{e_{/(1,0)}}$ are both isomorphisms, giving $\pi_{e_{/(1,0)},!}\operatorname{pr}^\ast\!x\cong x$. Therefore the object $(0,0)_!x$ does not vanish on $\partial V$. Hence $(0,0)_!$ does not have its image in $\D^{V,\partial V}\subset \D^V$.

We can generalise the above argument to other points in $V$. For any $(i,j)$ with $i,j>0$, $e_{/(i,j)}$ is still isomorphic to the terminal category, so $\pi_{e_{/(i,j),!}}\operatorname{pr}^\ast\!x\cong x$. For $i,j<0$, we have $(i,j)^\ast(0,0)_!x\cong 0$ because the slice category $e_{/(i,j)}$ is empty. Hence~$(0,0)_!x$ has the form
\begin{equation*}
\xymatrix@C=1em@R=1em{
&{}&{}&{}&{}\\
&&x\ar[r]&x\ar@.[ur]&{}\\
&0\ar[r]&x\ar[r]\ar[u]&x\ar[u]\\
&0\ar[r]\ar[u]&0\ar[u]\\
\ar@.[ur]&
}
\end{equation*}

However, if we compose $(0,0)_!$ with the localization $\D^V\to \D^{V,\partial V}$ we will necessarily obtain a left adjoint to $(0,0)^\ast$, as $(0,0)^\ast$ composed with the inclusion $\D^{V,\partial V}\to \D^V$ is still $(0,0)^\ast$. This means that something must happen on the $\partial V$ part of the diagram, but it is unclear from this description (slick as it is) what is actually going on at the level of objects. In order to understand better the left adjoint, the below proof is constructive.
\end{remark}

\begin{proof}
We construct the left adjoint in three stages. Let $V^{\leq 0}$ be the full subcategory of $V$ on those $(i,j)$ such that $i,j\leq 0$. Then consider the functor $(0,0)\colon e\to~V^{\leq0}$. This functor is a cosieve, so the left Kan extension $(0,0)_!\colon \D(e)\to \D(V^{\leq0})$ is extension by zero.

Next, let $V'$ be the full subcategory of $V$ which contains both $V^{\leq 0}$ and $\partial V$. That is, $V'$ contains all $(i,j)$ except for $(i,i)$ with $i>0$. Let $\iota_{\leq0}\colon V^{\leq0}\to V'$ be the inclusion. This map is a sieve, so the right Kan extension $\iota_{\leq0,\ast}$ is extension by zero. By Proposition~\ref{prop:stronglypointed}, $\iota_{\leq0,\ast}$ admits a right adjoint which we name $\iota_{\leq 0}^!$. Finally, we can consider the inclusion $\iota\colon V'\to V$. Therefore we obtain the following picture
\begin{equation*}
\vcenter{\xymatrix{
&\D\ar@/_1pc/[d]_{(0,0)_!}\ar `l/2em[dddl]`d[ddd]_L[ddd]\\
{}&\D^{V^{\leq0}}\ar@/_1pc/[d]_{\iota_{\leq0,\ast}}\ar@/_1pc/[u]_{(0,0)^\ast}\\
{}&\D^{V'}\ar@/_1pc/[d]_{\iota_!}\ar@/_1pc/[u]_{\iota_{\leq0}^!}\\
&\D^V\ar@/_1pc/[u]_{\iota^\ast}\\
}}
\end{equation*}

Let us call this left adjoint $L$. We will compute $L$ explicitly at the end of this section, and prove that for $x\in\D$, $Lx$ is the connective suspension prespectrum with $(Lx)_0\cong x$. This is then a derivator version of $\Sigma^\infty$. We now need to check that on $\D^{V,\partial V}$, the right adjoint to $L$ is $(0,0)^\ast\colon\D^{V,\partial V}\to~\D(e)$. To do this, we need to understand better the morphism~$\iota_{\leq 0}^!$.

From \cite[Corollary~3.8]{Gro13} we have that the $(\iota_{\leq 0,\ast},\iota_{\leq 0}^!)$ adjunction factors as
\begin{equation*}
\vcenter{\xymatrix{
\D^{V^{\leq 0}}\ar@/_1pc/[d]_{\iota_{\leq 0,\ast}}\\
\D^{V',V'\setminus V^{\leq 0}}\ar@/_1pc/[d]_{\text{incl}}\ar@/_1pc/[u]_{\iota_{\leq 0}^\ast}\\
\D^{V'}\ar@/_1pc/[u]_{R}
}}
\end{equation*}
where this bottom adjunction is a vanishing subderivator colocalization. Suppose that $X\in \D^{V,\partial V}$. Then $\iota^\ast X\in \D^{V'}$ already vanishes on $V'\setminus V^{\leq 0}$, so we know that the colocalization morphism $R$ is isomorphic to the identity on $\iota^\ast X$. Therefore for any $X\in \D^{V,\partial V}$,
\begin{equation*}
(0,0)^\ast\iota_{\leq 0}^!\iota^\ast X\cong (0,0)^\ast\iota_{\leq 0}^\ast \iota^\ast X=(0,0)^\ast X.
\end{equation*}
This shows that $L$ is left adjoint to $(0,0)^\ast$ on $\D^{V,\partial V}\subset\D^V$.

Therefore when we consider the corresponding morphism of derivators $L\colon\D\to~\Sp\D$, we obtain the total adjunction
\begin{equation}\label{eq:Ldiagram}
\vcenter{\xymatrix{
&\D\ar@/_1pc/[d]_{(0,0)_!}\ar `l/2em[dddl]`d[ddd]_L[ddd]&{}\\
{}&\D^{V^{\leq0}}\ar@/_1pc/[d]_{\iota_{\leq0,\ast}}\ar@/_1pc/[u]_{(0,0)^\ast}&\\
{}&\D^{V',V'\setminus V^{\leq 0}}\ar@/_1pc/[d]_{\iota_!}\ar@/_1pc/[u]_{\iota_{\leq0}^\ast}&\\
&\Sp\D\ar@/_1pc/[u]_{\iota^\ast}\ar `r/2em[uuur]`u[uuu]_{(0,0)^\ast}[uuu]\\
}}
\end{equation}
\end{proof}

Let us now compute the essential image of $L$ by going step by step. Let $x\in\D$. As indicated above, $\iota_{\leq0,\ast}(0,0)_!x$ extends $x$ by zero, giving underlying diagram
\begin{equation*}
\xymatrix@C=1em@R=1em{
&{}&{}&{}\\
&&0\ar@.[ur]&{}&{}\\
&0\ar[r]&x\ar[r]\ar[u]&0\ar@.[ur]\\
&0\ar[r]\ar[u]&0\ar[u]\\
\ar@.[ur]&
}
\end{equation*}
with a constant zero diagram towards the bottom left and zeroes along the diagonals going up and right. This shows (as we claimed above) that the image of $L$ vanishes on $\partial V$. The final functor $\iota_!$ is some sort of colimit. The next lemma will tell us what exactly it does.

\begin{lemma}\label{lemma:cocartsquares}
Let $i_n\colon \square\to V$ be the inclusion of the square $(a,b)\mapsto (a+n,b+n)$ for $n\geq 0$. Then $i_n^\ast Lx$ is cocartesian for any $x\in \D$.
\end{lemma}

We have not discussed the notion of cartesian and cocartesian squares in general in this paper. Such a discussion can be found at \cite[Definition~3.9]{Gro13} and the subsequent paragraphs. However, we only care about one type of cartesian and cocartesian square. Suppose we have the square
\begin{equation*}
\vcenter{\xymatrix@R=1em@C=1em{
0\ar[r]&b\\
a\ar[u]\ar[r]&0\ar[u]
}}
\end{equation*}
The square is cartesian if and only if the induced map $a\to \Omega b$ is an isomorphism (see Diagram~\ref{dia:comparisonmorphism}). This square is cocartesian if and only if the induced map $\Sigma a \to b$ is an isomorphism. These are the only types of (co)cartesian squares that will appear in this paper.

\begin{proof}
We use the following criterion from Groth. For $K\in\Dia$, define a \emph{square in $K$} to be a functor $i\colon\square\to K$ which is injective on objects.
\begin{prop}\label{prop:detectionlemma}{\cite[Proposition~3.10]{Gro13}}
Let $i\colon \square\to J$ be a square in $J$ and let $f\colon K\to J$ be a functor. Assume that the induced functor $\ll\to (J\setminus i(1,1))_{/i(1,1)}$ has a left adjoint and that $i(1,1)$ does not lie in the image of $f$. Then for all $Y\in\D(K)$, $X=f_!Y\in\D(J)$, the square $i^\ast X$ is cocartesian.
\end{prop}
For us, $J=V$, $K=V'$, $f$ is the inclusion $\iota$, and $i=i_n$. Let us investigate the indicated overcategory $(V\setminus~i_n(1,1))_{/i_n(1,1)}$. The objects of this category are objects $(i,j)\neq (n+1,n+1)$ along with a morphism $(i,j)\to (n+1,n+1)$. But since $V$ is a poset, we know that we either have a unique morphism $(i,j)\to (n+1,n+1)$ in the case that $i\leq n+1$ and $j\leq n+1$, or there is no such morphism. This makes the overcategory a subposet of our original poset $V'$. Therefore we introduce the notation
\begin{equation*}
V^{<n+1}:=\{(i,j)\in V:i\leq n+1,j\leq n+1, (i, j)\neq(n+1,n+1)\}
\end{equation*}
for the overcategory we wish to investigate. We now examine the functor \linebreak$r\colon \ll\to V^{<n+1}$ and construct a left adjoint $\ell$ to it explicitly.

Let $(i,j)$ denote an arbitrary object of $V^{<n+1}$, and let $(a,b)$ denote an arbitrary object of $\ll$. Then we would like
\begin{equation*}
\Hom_{V^{<n+1}}((i,j),r(a,b))\cong \Hom_{\ll}(\ell(i,j),(a,b)).
\end{equation*}
We know that $r(i,j)=(a+n,b+n)$, so we can characterise this hom-set as follows:
\begin{equation*}
\Hom_{V^{<n+1}}((i,j),r(a,b))=\begin{cases}
\ast&i\leq a+n\text{ and }j\leq b+n\\
\varnothing&i>a+n\text{ or }j>b+n
\end{cases}
\end{equation*}

If $i\leq n$ and $j\leq n$, we will be in the top case, so we will set $\ell(i,j)=(0,0)$ in that case. Since $(0,0)$ is initial in $\ll$, we know that there is always a unique morphism to $(a,b)$ no matter what.

Suppose that $(i,j)=(n,n+1)$. Then $\Hom_{V^{<n+1}}((n,n+1),(a+n,b+n))$ is nonempty if and only if $b=1$. That means that $\ell(n,n+1)$ should have a morphism to $(0,1)$ but not to $(1,0)$. This forces $\ell(n,n+1)=(0,1)$. Similarly, we can see that we need $\ell(n+1,n)=(1,0)$. This defines $\ell$ on every object and establishes a bijection of hom-sets. As this bijection is necessarily natural in both variables, we obtain the adjunction.

For any square $i_n\colon \square\to V$ as described above, we have that $i_n(1,1)=(n+1,n+1)$ is not in the image of $\iota\colon V'\to V$. Therefore the proposition applies, showing that $i_n^\ast \iota_! Y$ is cocartesian for any $n\in\N$ and any $Y\in\D(V')$. This in particular applies to $i_n^\ast Lx$, for $x\in\D$, completing the proof of the lemma.
\end{proof}

We can now write the underlying diagram of $Lx$. It is zero almost everywhere, except on the entires $(i,i)$ where $i\geq 0$. Write $X_i$ for $(i,i)^\ast Lx$, and note that $X_0\cong x$. We first restrict to the square $i_0\colon \square\to V$:
\begin{equation*}
\xymatrix@C=1em@R=1em{
0\ar[r]&X_1\\
x\ar[r]\ar[u]&0\ar[u]
}
\end{equation*}
Because this square is cocartesian, we have $X_1\cong \Sigma x$. Iterating this process, we see that $X_i\cong \Sigma^ix$. Therefore as we claimed above, $L\colon \D\to\Sp\D$ gives a (connective) suspension prespectrum on whatever object we start with.

\section{Stabilization}

We now want to give a formal definition of the subderivator of stable spectra in $\Sp\D$. To do so, we need a better definition than `all the structure maps are isomorphisms'. This is not too hard to do, but it does require some new diagram notation.

\begin{notn}
Recall the notation for the category $V$ of Definition~\ref{defn:prespectrum}. Let \linebreak$\sigma\colon V\to V$ be the functor defined by $(i,j)\mapsto (j+1,i+1)$. Since $\sigma(\partial V)\subset\partial V$ and $\sigma^\ast$ is a pointed morphism, we obtain an endomorphism $\sigma^\ast\colon \Sp\D\to\Sp\D$.
\end{notn}

\begin{notn}
For any diagram $J$, let $\Sp\D^J$ denote the derivator $\Sp\D$ shifted by~$J$, not the unstable spectrum derivator associated to $\D^J$. These are, in fact, exactly the same derivator, but we prefer the former interpretation for the following constructions.
\end{notn}

\begin{cons}\label{cons:w}
There exists a functor $w\colon \square\times V\to V$ such that $w^\ast$ restricts to a morphism $\Sp\D\to\Sp\D^\square$ and, for any $X\in\Sp\D$, we have
\begin{equation*}
\dia_\square(w^\ast X)\,=\vcenter{
\xymatrix@C=1em@R=1em{
0\ar[r]&\sigma^\ast X\\
X\ar[u]\ar[r]&0\ar[u]
}}
\end{equation*}
\end{cons}
We draw inspiration from Heller's construction of $w$ in \cite[\S8]{Hel97}. Throughout, $(i,j)$ will denote an object of $V$ and $(a,b)$ will denote an object of $\square$ (recall Notation~\ref{notn:square}). Let $\tau\colon\square\to\square$ be defined by $(a,b)\mapsto(b,a)$. We would like $w$ to satisfy the following three properties:

\begin{enumerate}
\item $w(a,b,i,i)=(i+a,i+b)$.
\item $w(a,b,i,j)\in\partial V$ for $i\neq j$.
\item $w\circ(\tau\times\,\sigma)=\sigma\circ w$.
\end{enumerate}

We will give the definition of $w$ first, followed by an illustration.
\begin{equation}\label{eq:wdef}
w(a,b,i,j)=\begin{cases}
(i,j)&(a,b)=(0,0)\\
(j+1,i+1)&(a,b)=(1,1)\\
(j+1,j)&(a,b)=(1,0)\\
(i,i+1)&(a,b)=(0,1)
\end{cases}
\end{equation}
To see this in action, let us restrict to the subset of $V$ where $i,j\in\{0,1,2,3\}$. Then the codomain of $w$ (given this restriction) we may write as
\begin{equation}\label{dia:codomainletters}
\vcenter{\xymatrix@C=1em@R=1em{
&& \iota\ar[r] & \kappa\\
& \zeta\ar[r] & \eta\ar[u]\ar[r] & \theta\ar[u]\\
\gamma\ar[r] & \delta\ar[r]\ar[u] & \varepsilon\ar[u]  &\\
\alpha\ar[r]\ar[u] & \beta\ar[u]&
}}
\end{equation}
where $\alpha=(0,0)$ and $\kappa=(3,3)$. We will restrict the domain to the subset of $V$ where $i,j\in\{0,1,2\}$ and write at each place $(a,b,i,j)\in\square\times V$ the element $w(a,b,i,j)\in V$ in Diagram~\ref{dia:codomainletters}. This gives us the underlying diagram of $w^\ast X$ for a coherent object $X\in\D^V$ with incoherent diagram as above. The starting picture, without decoration yet, is this:

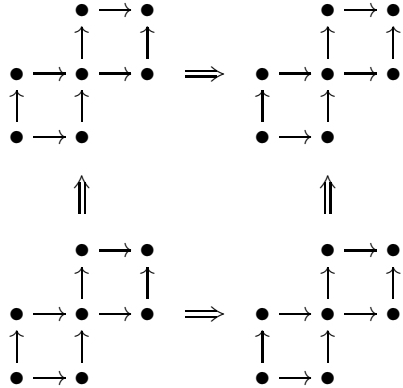
\begin{equation*}
\xymatrix@R=1em@C=1em{
& \bullet\ar[r]& \bullet & {} & & \bullet\ar[r] & \bullet &\\
\bullet\ar[r] & \bullet\ar[r]\ar[u] & \bullet\ar[u] \ar@{}[rr]|(.33){}="A0"|(.66){}="A1"& {} & \bullet\ar[r] & \bullet\ar[r]\ar[u] & \bullet\ar[u]\\
\bullet\ar[r]\ar[u] & \bullet\ar[u]& & {} & \bullet\ar[r]\ar[u] & \bullet\ar[u]&\\
{}\\
& \bullet\ar[r]\ar@{}[uu]|(.33){}="C0"|(.66){}="C1" & \bullet & {} & & \bullet\ar[r]\ar@{}[uu]|(.33){}="D0"|(.66){}="D1" & \bullet &\\
\bullet\ar[r]& \bullet\ar[r]\ar[u] & \bullet\ar[u] \ar@{}[rr]|(.33){}="B0"|(.66){}="B1"& {} & \bullet\ar[r] & \bullet\ar[r]\ar[u] & \bullet\ar[u]\\
\bullet\ar[r]\ar[u] & \bullet\ar[u]& & {} & \bullet\ar[r]\ar[u] & \bullet\ar[u]&
\ar@{=>} "A0";"A1" \ar@{=>} "B0";"B1" \ar@{=>} "C0";"C1" \ar@{=>} "D0";"D1"
}
\end{equation*}

The bold arrows $\Rightarrow$ on each side represent the $\square$ dimension of the diagram, which connects the coherent subdiagrams in each corner. They do not define a morphism between the objects they might seem to, but orient the diagram as a whole. This means the lower left corner is $(0,0,0,0)$ and the upper right corner is $(1,1,2,2)$.

Criterion (1) fixes what happens along the diagonal of each corner of the square. The upper right and lower left corners of the square should have underlying diagrams $X$ and $\sigma^\ast X$ (respectively). We place these objects in the diagram and mark some more for further study.

\begin{equation*}
\vcenter{\xymatrix@R=1em@C=1em{
& \vartriangle_1\ar[r]& \iota & {} & & \theta\ar[r] & \kappa &\\
\bullet\ar[r] & \zeta\ar[r]\ar[u] & \star_1\ar[u] \ar@{}[rr]|(.33){}="A0"|(.66){}="A1"& {} & \varepsilon\ar[r] & \eta\ar[r]\ar[u] & \iota\ar[u]\\
\gamma\ar[r]\ar[u] & \bullet\ar[u]& & {} & \delta\ar[r]\ar[u] & \zeta\ar[u]&\\
{}\\
& \zeta\ar[r]\ar@{}[uu]|(.33){}="C0"|(.66){}="C1" & \eta & {} & & \star_2\ar[r]\ar@{}[uu]|(.33){}="D0"|(.66){}="D1" & \theta &\\
\gamma\ar[r]& \delta\ar[r]\ar[u] & \varepsilon\ar[u] \ar@{}[rr]|(.33){}="B0"|(.66){}="B1"& {} & \bullet\ar[r] & \varepsilon\ar[r]\ar[u] & \vartriangle_2\ar[u]\\
\alpha\ar[r]\ar[u] & \beta\ar[u]& & {} & \beta\ar[r]\ar[u] & \bullet\ar[u]&
\ar@{=>} "A0";"A1" \ar@{=>} "B0";"B1" \ar@{=>} "C0";"C1" \ar@{=>} "D0";"D1"
}}
\end{equation*}

Let us now examine the objects marked $\star$. We see that $\star_1$ must satisfy $\zeta\leq w(\star_1)$, $\varepsilon\leq w(\star_1)$, and $w(\star_1)\leq \iota$. By criterion (2), we must have $w(\star_1)\neq \eta$. This forces $w(\star_1)=\iota$. Similarly, $w(\star_2)=\theta$.

For the objects marked $\vartriangle$, we have $\zeta\leq w(\vartriangle_1)$, $w(\vartriangle_1)\leq \theta$, and $w(\vartriangle_1)\leq \iota$. Because $w(\vartriangle_1)\neq \eta$ by criterion (2), this forces $w(\vartriangle_1)=\zeta$. Similarly, $w(\vartriangle_2)=\varepsilon$. We can then define $w$ on the unadorned $\bullet$ objects by using criterion (3). Therefore we complete the picture:

\begin{equation*}
\vcenter{\xymatrix@R=1em@C=1em{
& \zeta\ar[r]& \iota & {} & & \theta\ar[r] & \kappa &\\
\gamma\ar[r] & \zeta\ar[r]\ar[u] & \iota\ar[u] \ar@{}[rr]|(.33){}="A0"|(.66){}="A1"& {} & \varepsilon\ar[r] & \eta\ar[r]\ar[u] & \iota\ar[u]\\
\gamma\ar[r]\ar[u] & \zeta\ar[u]& & {} & \delta\ar[r]\ar[u] & \zeta\ar[u]&\\
{}\\
& \zeta\ar[r]\ar@{}[uu]|(.33){}="C0"|(.66){}="C1" & \eta & {} & & \theta\ar[r]\ar@{}[uu]|(.33){}="D0"|(.66){}="D1" & \theta &\\
\gamma\ar[r]& \delta\ar[r]\ar[u] & \varepsilon\ar[u] \ar@{}[rr]|(.33){}="B0"|(.66){}="B1"& {} & \varepsilon\ar[r] & \varepsilon\ar[r]\ar[u] & \varepsilon\ar[u]\\
\alpha\ar[r]\ar[u] & \beta\ar[u]& & {} & \beta\ar[r]\ar[u] & \beta\ar[u]&
\ar@{=>} "A0";"A1" \ar@{=>} "B0";"B1" \ar@{=>} "C0";"C1" \ar@{=>} "D0";"D1"
}}
\longrightarrow\quad
\vcenter{\xymatrix@C=1em@R=1em{
&& \iota\ar[r] & \kappa\\
& \zeta\ar[r] & \eta\ar[u]\ar[r] & \theta\ar[u]\\
\gamma\ar[r] & \delta\ar[r]\ar[u] & \varepsilon\ar[u]  &\\
\alpha\ar[r]\ar[u] & \beta\ar[u]&
}}
\end{equation*}

As claimed, we have $(0,0)^\ast w^\ast X=X$ and $(1,1)^\ast w^\ast X=\sigma^\ast X$. Criterion (2) guarantees that $w(\square\times\partial V)\subset\partial V$, so the pointed morphism $w^\ast$ restricts to $\Sp\D$.

\begin{remark}\label{rk:wrongw}
Heller at \cite[p.127]{Hel97} uses a slightly different definition for $\sigma$, namely the `obvious' shift $s(i,j)=(i+1,j+1)$. He also replaces $\tau$ by $\id_\square$ in criterion~(3). However, there is no way to construct $w$ such that $(0,0)^\ast w^\ast X=X$ and $(1,1)^\ast w^\ast X=~s^\ast X$ satisfying $w\circ(\id_\square\times s)=s\circ w$ despite Heller's assertion (without proof) to the contrary.

To give a quick proof, restrict to $i,j\in\{0,1,2\}$. Then Heller's criterion (1) forces the following, where we emphasise the entry $(0,1,1,0)$ in particular:
\begin{equation*}
\vcenter{\xymatrix@R=1em@C=1em{
\bullet\ar[r]&\zeta&\zeta\ar[r]&\eta\\
\gamma\ar[r]_-{}="A1"\ar[u]&\star\ar[u]_-{}="B0"&\delta\ar[r]_-{}="D1"\ar[u]^-{}="B1"&\varepsilon\ar[u]\\
\gamma\ar[r]^-{}="A0"&\delta&\bullet\ar[r]^-{}="D0"&\varepsilon\\
\alpha\ar[r]\ar[u]&\beta\ar[u]_-{}="C0"&\beta\ar[r]\ar[u]^-{}="C1"&\bullet\ar[u]
\ar@{}"A0";"A1"|(.33){}="a0"|(.66){}="a1" \ar@{=>}"a0";"a1"
\ar@{}"B0";"B1"|(.33){}="b0"|(.66){}="b1" \ar@{=>}"b0";"b1"
\ar@{}"C0";"C1"|(.33){}="c0"|(.66){}="c1" \ar@{=>}"c0";"c1"
\ar@{}"D0";"D1"|(.33){}="d0"|(.66){}="d1" \ar@{=>}"d0";"d1"
}}\quad\longrightarrow\quad
\vcenter{\xymatrix@C=1em@R=1em{
&& \iota\ar[r] & \kappa\\
& \zeta\ar[r] & \eta\ar[u]\ar[r] & \theta\ar[u]\\
\gamma\ar[r] & \delta\ar[r]\ar[u] & \varepsilon\ar[u]  &\\
\alpha\ar[r]\ar[u] & \beta\ar[u]&
}}
\end{equation*}
We know that $w(\star)$ must satisfy $\beta\leq w(\star)\leq \varepsilon$ and $\gamma\leq w(\star)\leq \zeta$. This forces $w(\star)=\delta$, which contradicts criterion (2) that $w(\square\times\partial V)\subset\partial V$.

Note that on objects $X\in\Sp\D$, $s^\ast X$ and $\sigma^\ast X$ appear identical incoherently, though coherently they differ (subtly) because of the twist. Therefore we may continue following Heller's reasoning with this proper construction of $w$.
\end{remark}

What does this construction accomplish? Recall from the discussion following Definition~\ref{defn:prespectrum} that from the coherent square
\begin{equation}\label{dia:constructionofphi}
\vcenter{\xymatrix@C=1em@R=1em{
0\ar[r]&\sigma^\ast X\\
X\ar[u]\ar[r]&0\ar[u]
}}
\end{equation}
we obtain a canonical coherent map $(X\to\Omega\sigma^\ast X)$ that we will call $\phi_X$. This process is natural in $X$ and gives a morphism of derivators $\phi\colon \Sp\D\to\Sp\D^{[1]}$.

\begin{defn}\label{defn:spectrum}
An object $X\in\Sp\D$ is called an \emph{$\Omega$-spectrum} if the canonical map $\phi_X\colon X\to\Omega\sigma^\ast X$ is an isomorphism.
\end{defn}
Note that $(n,n)^\ast\phi_X$ is exactly the $n$th structure map of $X$, and all the structure maps are isomorphisms if and only if $\phi_X$ is by Der2.

Let $\cS_K\subset\Sp\D(K)$ be the full subcategory comprised of the $\Omega$-spectra. Note that each subcategory $\cS_K$ is replete, as an isomorphism $X\to Y$ in $\Sp\D(K)$ will give an isomorphism $\phi_X\to \phi_Y$ of coherent maps, so $X$ is an $\Omega$-spectrum if and only if $Y$ is.

Further, note that $\cS_K$ assemble to a prederivator $\St\D$. Suppose that \linebreak$X\in\St\D(K)$, and let $u\colon J\to K$ be a functor. Then because $\phi\colon \Sp\D\to\Sp\D^{[1]}$ is a morphism of derivators, $\phi_{u^\ast X}\cong u^\ast\phi_X$. Because $\phi_X$ is an isomorphism, so is $u^\ast\phi_X$ and hence $u^\ast X$ is an $\Omega$-spectrum as well.

Before stating the main theorem of this section, we need to define the second property of a derivator necessary for stabilization and prove a small lemma about it.

\begin{defn}
A derivator $\D$ is called \emph{regular} if filtered colimits commute with finite limits. That is, for any filtered category $C$ and finite category $D$, the canonical morphism
\begin{equation}\label{eqn:regular}
\pi_{C,!}(\id_C\times\pi_D)_\ast Y\to \pi_{D,\ast}(\pi_C\times\id_D)_! Y
\end{equation}
is an isomorphism for any $Y\in\D(C\times D)$, where $\pi_C,\pi_D$ are the projections to the final category.
\end{defn}

In a stable derivator, all colimits commute with finite limits; see \cite[Theorem~3.15]{Gro16b}. Hence we can view regularity as a test of pre-stability. Requiring that our derivators be regular is not too restrictive in practice. This condition is satisfied for an enormous class of examples, broadest of all being locally presentable categories; see \cite[Proposition~1.59]{AdaRos94}. The statement of that proposition is for directed colimits, but the corollary to Theorem~1.5 in that reference shows there is no distinction. The same is true generally in any $n$-topos, for $n\in\N$ or $n=\infty$; see Example~7.3.4.7 in combination with the dual of Proposition~5.3.2.9 in \cite{Lur09}.

We will only need that sequential colimits commute with pullbacks for Lemma~\ref{lemma:regular}, and this is (more or less) Heller's original definition at \cite[\S5]{Hel88}. But there is no distinction in examples between sequential and filtered colimits and between pullbacks and finite limits. We pick this definition not because we must at the moment, but for a future application of stabilization to derivator K-theory, which we anticipate will require the stronger axiom.

\begin{prop}\label{prop:localizeregular}
Let $\D$ be a regular pointed derivator. Then $\Sp\D$ is also a regular pointed derivator.
\end{prop}
\begin{proof}
First, if $\D$ is regular, so is $\D^V$. Let us keep the notation of Equation~\ref{eqn:regular} but suppress $\id_C$ and $\id_D$ for simplicity. The comparison morphism for $Y\in~\D^V(C\times D)$ gives rise to an isomorphism after applying $(i,j)^\ast$ for any $(i,j)\in V$
\begin{equation*}
\vcenter{\xymatrix{
(i,j)^\ast\pi_{C,!}\pi_{D,\ast} Y\ar[r]^-\cong&
\pi_{C,!}\pi_{D,\ast}(i,j)^\ast Y\ar[r]^-\cong&
\pi_{D,\ast}\pi_{C,!} (i,j)^\ast Y\ar[r]^-\cong&
(i,j)^\ast \pi_{D,\ast}\pi_{C,!} Y,
}}
\end{equation*}
where the first and last isomorphisms follow from the bicontinuity of $(i,j)^\ast$ and the middle isomorphism from the regularity of $\D$. This means that the comparison morphism for $\D^V$ is a pointwise isomorphism, hence must be an isomorphism by Der2.

Now consider the inclusion $\iota\colon\Sp\D\to\D^V$. We use critically that a vanishing subderivator localization is also a colocalization, so that the inclusion preserves both limits and colimits. For $X\in\Sp\D(C\times D)$, we need to show that the below square commutes up to isomorphism:
\begin{equation*}
\vcenter{\xymatrix{
\Sp\D(C\times D)\ar[r]^-{\pi_{D,\ast}}\ar[d]_-{\pi_{C,!}}&\Sp\D(C)\ar[d]^-{\pi_{!,C}}\ar@{}[dl]|\swtrans\\
\Sp\D(D)\ar[r]_-{\pi_{D,\ast}}&\Sp\D(e)
}}
\end{equation*}

Using $\iota$, we can build a cube involving the derivator $\D^V$ (omitting the natural transformations on each face):
\begin{equation*}
\vcenter{\xymatrix@R=1em@C=1em{
&\D^V(C\times D)\ar[rr]\ar[dd]|!{[dr];[dl]}\hole&&\D^V(C)\ar[dd]\\
\Sp\D(C\times D)\ar[ur]^-{\iota_{C\times D}}\ar[rr]^(0.66){\pi_{D,\ast}}\ar[dd]_-{\pi_{C,!}}&&\Sp\D(C)\ar[ur]^-{\iota_C}\ar[dd]^(.33){\pi_{!,C}}\\
&\D^V(D)\ar[rr]|!{[ur];[dr]}\hole&&\D^V(e)\\
\Sp\D(D)\ar[ur]^-{\iota_D}\ar[rr]_-{\pi_{D,\ast}}&&\Sp\D(e)\ar[ur]_-{\iota_e}
}}
\end{equation*}
where the horizontal and vertical functors on the back face are the same Kan extensions as those on the front face. The back face commutes up to isomorphism because $\D^V$ is regular. Moreover, the top, bottom, left, and right faces commute up to isomorphism because $\iota$ is bicontinuous. This implies that the front face commutes up to isomorphism after applying $\iota_e$. But since this functor is fully faithful, the front face must already commute up to isomorphism and so we conclude that $\Sp\D$ is regular.
\end{proof}

\begin{theorem}\label{thm:stdlocalisation}
Let $\D$ be a regular pointed derivator. Then inclusion $i\colon\St\D\to~\Sp\D$ admits a left adjoint $\loc\colon\Sp\D\to\St\D$. That is, $\St\D$ is a localization of $\Sp\D$.
\end{theorem}

We will give the proof of this theorem in a series of constructions, with the ultimate goal of an explicit formula for the localization. Heller again gives us an idea for what the adjoint $\loc\colon \Sp\D\to\St\D$ should be. We are looking for something of the form
\begin{equation}\label{eq:loc}
\loc X=\operatorname{hocolim}(X\to\Omega s^\ast X\to \Omega^2(s^\ast)^2X\to\cdots),
\end{equation}
where $s\colon V\to V$ is Heller's incorrect shift functor of Remark~\ref{rk:wrongw}. Heller gives this as a definition, but offers no way to obtain the \emph{coherent} diagram of which we would like to take the homotopy colimit. The content of the following construction is to repair these errors.

We know that $\phi_X$ gives us, more or less, the first stage of the homotopy colimit. We are free to iterate this process, and in fact for any $n\in\N$ we can obtain a coherent diagram
\begin{equation*}
X\to \Omega\sigma^\ast X\to (\Omega\sigma^\ast)^2X\to\cdots\to(\Omega\sigma^\ast)^nX.
\end{equation*}
However, we cannot end up with an infinite coherent diagram like the one we need though this induction.

\begin{cons}\label{cons:coherentdia}
There is a morphism of derivators $F\colon \Sp\D\to\Sp\D^\N$ given incoherently by
\begin{equation*}
X\mapsto (X\to\Omega\sigma^\ast X\to \Omega^2(\sigma^\ast)^2X\to\cdots).
\end{equation*}
\end{cons}
We will modify our functor $w\colon \square\times V\to V$ from Construction~\ref{cons:w} to give us an infinite diagram of the form we want. To begin with, let us define the subcategory $V^{\geq 0}\subset V$ of all $(i,j)$ such that $i,j\geq 0$. We can see $\square\subset V^{\geq 0}$ by using the obvious embedding of both in $\N^2$. We want to define a functor $\omega\colon V^{\geq 0}\times V\to V$ with the aim that, for every square $i_n\colon \square\to V^{\geq 0}$ given by $(a,b)\mapsto (a+n,b+n)$,
\begin{equation*}
(i_n)^\ast \omega^\ast X=(\sigma^\ast)^nw^\ast X.
\end{equation*}
That is, $\omega^\ast$ should be an infinite version of $w^\ast$.

Note that, by the cartesian closed structure on $\Cat$, the functor $w\colon \square\times V\to V$ corresponds canonically to a functor $V\to V^\square$. Specifically,
\begin{equation*}
(i,j)\mapsto w(-,-,i,j).
\end{equation*}
Rather than define $\omega$ as we did in Equation~\ref{eq:wdef}, we will define the functor $V\to V^{V^{\geq 0}}$ that corresponds to it. As we did above, we will write the functor $V^{\geq0}\to V$ by giving a diagram of shape $V^{\geq 0}$ with entries given by the value of the functor. To demonstrate this for $w$, our alternative definition $V\to V^\square$ is
\begin{equation*}
(i,j)\mapsto \vcenter{\xymatrix@R=1em@C=1em{
(i,i+1)\ar[r]&(j+1,i+1)\\
(i,j)\ar[r]\ar[u]&(j+1,j)\ar[u]
}}
\end{equation*}
so that $w\sigma^n$ for $n$ odd is defined by
\begin{equation*}
(i,j)\mapsto \vcenter{\xymatrix@R=1em@C=1em{
(j+n,j+n+1)\ar[r]&(i+n+1,j+n+1)\\
(j+n,i+n)\ar[r]\ar[u]&(i+n+1,i+n)\ar[u]
}}
\end{equation*}
and $w\sigma^n$ for $n$ even is defined by
\begin{equation*}
(i,j)\mapsto \vcenter{\xymatrix@R=1em@C=1em{
(i+n,i+n+1)\ar[r]&(j+n+1,i+n+1)\\
(i+n,j+n)\ar[r]\ar[u]&(j+n+1,j+n)\ar[u]
}}
\end{equation*}

Without further ado, $\omega$ will be defined by
\begin{equation*}
(i,j)\mapsto
\vcenter{\xymatrix@R=1em@C=1em{
&&&&{}\\
&&(i+2,i+3)\ar[r]&(j+3,i+3)\ar@.[ur]\\
&(j+1,j+2)\ar[r]&(i+2,j+2)\ar[r]\ar[u]&(j+3,j+2)\ar[u]\\
(i,i+1)\ar[r]&(j+1,i+1)\ar[u]\ar[r]&(i+2,i+1)\ar[u]\\
(i,j)\ar[r]\ar[u]&(j+1,j)\ar[u]
}}
\end{equation*}

It may be apparent why we chose this format rather than the other definition. Examination of the above diagram (or direct calculation) shows
\begin{equation*}
(i_n)^\ast\omega^\ast X\cong (\sigma^\ast)^n w^\ast X
\end{equation*}
as we had hoped.

We now need to construct the $\N$-shaped coherent morphism suggested by Equation~\ref{eq:loc}. First, we need to embed $\omega^\ast X$ into a larger diagram shape. Let $Z$ be the full subcategory of $V^{\geq 0}\times \N$ consisting of all $(i,j,k)$ for $i\neq j$ but restricting to $k\leq i$ if $i=j$.

There is a projection $\zeta\colon Z\to V^{\geq 0}$ that forgets the $k$ component. For $X\in\Sp\D$, we have that
\begin{equation*}
(i,j,k)^\ast\zeta^\ast \omega^\ast X=(i,j)^\ast\omega^\ast X
\end{equation*}
for any $k\in\N$, so $\zeta^\ast$ gives us a constant diagram functor of variable lengths depending on the value of $(i,j)$. To illustrate this, we draw $\zeta^\ast\omega^\ast X$ for $i,j,k\leq 2$:
\begin{equation}\label{dia:Z}
\vcenter{\xymatrix@R=1ex@C=1ex{
& & & & & & & {}\\
& & & & 0\ar[rr]\ar'[d][dd] & & (\sigma^\ast)^2 X\ar[dd]\ar@.[ur]\\
& 0\ar[rr]\ar'[d][dd] & & \sigma^\ast X\ar[dd]\ar[rr]\ar[ur] & & 0\ar[ur]\ar[dd]&&{}\\
X\ar[ur]\ar[rr] & & 0\ar[ur]\ar[dd] & & 0\ar'[r][rr]\ar'[d][dd] & & (\sigma^\ast)^2 X\ar[dd]\ar@.[ur]\\
& 0\ar'[r][rr]\ar[dd] & & \sigma^\ast X\ar[ur]\ar[rr] & & 0\ar[ur]\ar[dd]&&{}\\
& & 0\ar[ur]\ar[dd] & & 0\ar'[r][rr]\ar@.[d] & & (\sigma^\ast)^2 X\ar@.[ur]\\
& 0\ar[rrru]|!{[ur];[dr]}{\hole}\ar@.[d] & & & & 0\ar[ur]\ar@.[d]\\
& & 0\ar[rrru]\ar@.[d] & & & &\\
& & & & & & &
}}
\end{equation}

From this complicated shape, we will attempt to find our desired coherent diagram. Let $f\colon Z\to V^{\geq 0}\times \N$ be the natural embedding. The last step of our construction will be to take the right Kan extension $f_\ast$ and restrict to the subcategory $\{(0,0)\}\times \N$, which should give us the desired coherent diagram.

We will again use Groth's detection lemma to make sure that we obtain cartesian squares where we would like them. We state the dual of Proposition~\ref{prop:detectionlemma} for reference:
\begin{prop}\label{prop:cartdetectlemma}
Let $i\colon \square\to J$ be a square in $J$ and let $f\colon K\to J$ be a functor. Assume that the induced functor $\ur\to (J\setminus i(0,0))_{i(0,0)/}$ has a right adjoint and that $i(0,0)$ does not lie in the image of $f$. Then for all $X=f_\ast Y\in \D(J)$, $Y\in\D(K)$, the square $i^\ast X$ is cartesian.
\end{prop}

In our notation, we have $J=V^{\geq 0}\times \N$ and $K=Z$. The squares that we care about being cartesian are at a constant $k$-level, so let us use the notation $i_{n,l}\colon \square\to V^{\geq 0}\times\N$ for the inclusion $(a,b)\mapsto (a+n,b+n,l)$ where $l>n$. Let us examine the category $(V^{\geq 0}\times\N\setminus i_{n,l}(0,0))_{i_{n,l}(0,0)/}$. Because $V^{\geq 0}\times\N$ is a poset, this slice category will be a subcategory. In particular, it will be $(i,j,k)\in V^{\geq 0}\times\N$ such that $i,j\geq n$ but $(i,j)\neq (n,n)$ and $k\geq l$. Let us call this subcategory $Z^{> n}$.

We will construct the right adjoint $r$ directly, similar to the proof of Lemma~\ref{lemma:cocartsquares}. Let $(a,b)\in \ur$ and let $(i,j,k)\in Z^{>n}$. Call the induced functor $\ell\colon \ur\to Z^{>n}$. Then we want
\begin{equation*}
\Hom_{Z^{>n}}(\ell(a,b),(i,j,k))\cong\Hom_\ur((a,b),r(i,j,k)).
\end{equation*}
We know that 
\begin{equation*}
\Hom_{Z^{>n}}(\ell(a,b),(i,j,k))=\begin{cases}
\ast& i \geq a+n, j\geq b+n,\text{ and }k\geq l \\
\varnothing& \text{otherwise}
\end{cases}
\end{equation*}
The $k$ coordinate is irrelevant, since $k\geq l$ is always satisfied. Because $(a,b)=(1,1)$ is final in $\ur$, if $i\geq n+1$ and $j\geq n+1$, we must have $r(i,j,k)=(1,1)$. Therefore we only need to consider $(n,n+1,k)$ and $(n+1,n,k)$. We see that $\ell(0,1)$ has a map to $(n,n+1,k)$ but not to $(n+1,n,k)$, so we must have $r(n,n+1,k)=(0,1)$, and similarly $r(n+1,n,k)=(1,0)$. This gives us the desired right adjoint, as again naturality of the hom-set bijection is automatic for posets.

Having established that all squares $i_{n,l}^\ast f_\ast \zeta^\ast\omega^\ast X$ are cartesian, we should see what this gives us in cash. Let $\widetilde X=f_\ast \zeta^\ast\omega^\ast X$ for ease of notation. If we look particularly at the cartesian square given by $i_{n-1,n}^\ast\widetilde X$,
\begin{equation*}
\xymatrix@C=1em@R=1em{
0\ar[r]&(\sigma^\ast)^{n}X\\
Y_1\ar[r]\ar[u]&0\ar[u]
}
\end{equation*}
we have $Y_1\cong \Omega(\sigma^\ast)^{n}X$. We can now examine the square $i_{n-2,n}^\ast\widetilde X$. Because we have established that $(n-1,n-1,n)^\ast\widetilde X\cong \Omega(\sigma^\ast)^n X$, this square is of the form
\begin{equation*}
\xymatrix@C=1em@R=1em{
0\ar[r]&\Omega(\sigma^\ast)^{n}X\\
Y_2\ar[r]\ar[u]&0\ar[u]
}
\end{equation*}
so that $Y_2\cong \Omega^2(\sigma^\ast)^n X$. Repeating this process shows that
\begin{equation*}
(0,0,n)^\ast\widetilde X\cong \Omega^n(\sigma^\ast)^n X.
\end{equation*}

The last step is to restrict to the column $\{(0,0,k)\}\subset V^{\geq 0}\times \N$, which is precisely the diagram we need. Hence we obtain coherently
\begin{equation}\label{eq:coherentomega}
X\to \Omega\sigma^\ast X\to \Omega^2(\sigma^\ast)^2X\to\cdots\to\Omega^n(\sigma^\ast)^nX\to\cdots
\end{equation}
where each of these maps is a shift of the map $\phi_X$ we developed for Definition~\ref{defn:spectrum}.

Recall that we named the functor sending $X$ to the diagram of Equation~\ref{eq:coherentomega} $F\colon\Sp\D\to\Sp\D^\N$. We let $\loc\colon \Sp\D\to\Sp\D$ be given by the composition
\begin{equation}\label{eq:locdefinition}
\Sp\D\overset{\omega^\ast}\longrightarrow
\Sp\D^{V^{\geq 0}}\overset{\zeta^\ast}\longrightarrow
\Sp\D^Z \overset{f_\ast}\longrightarrow
\Sp\D^{V^{\geq 0}\times\N}\overset{(0,0)^\ast}\longrightarrow
\Sp\D^\N\overset{\pi_!}\longrightarrow
\Sp\D,
\end{equation}
so that $\loc=\pi_!F$, where $\pi\colon\N\to e$ is the projection. We now need to make sure $\loc$ is a good start for our localization.

\begin{lemma}\label{lemma:regular}
The morphism $\loc\colon \Sp\D\to\Sp\D$ has essential image $\St\D$.
\end{lemma}
\begin{proof}
Let $X\in\Sp\D$. We need to show that $\phi_{\loc X}\colon \loc X\to \Omega\sigma^\ast \loc X$ is an isomorphism. We do this by (incoherently) showing that it is a composition of isomorphisms:
\begin{equation*}
\vcenter{\xymatrix@C=3em{
\loc X\ar[r]^-{\loc (\phi_X)}\ar@(d,l)[dr]_{\phi_{\loc X}}&\loc\Omega\sigma^\ast X\ar[d]\\
&\Omega\sigma^\ast \loc X
}}
\end{equation*}
where the vertical map is the canonical comparison map.

We first prove that $\loc(\phi_X)\colon\loc X\to\loc\Omega\sigma^\ast X$ is an isomorphism. To first give the picture of $F\phi_X$ in $\Sp\D(\N\times[1])$,
\begin{equation*}
\vcenter{\xymatrix{
X\ar[r]\ar[d]&\Omega\sigma^\ast X\ar[r]\ar[d]&\Omega^2(\sigma^\ast)^2X\ar[r]\ar[d]&\cdots\\
\Omega\sigma^\ast X\ar[r]&\Omega^2(\sigma^\ast)^2 X\ar[r]&\Omega^3(\sigma^\ast)^3X\ar[r]&\cdots
}}
\end{equation*}
The horizontal morphisms and the vertical morphisms with the same domain and codomain are precisely the same map. Therefore $\loc(\phi_X)$ is the shift map between two sequential colimits, which is an isomorphism as the value of the sequential colimit does not depend on what finite stage we start at.

We now need to prove that $\loc$ commutes with $\Omega\sigma^\ast$, and do this we must look at the construction of $\loc$. Because $\sigma^\ast$ is a pullback morphism, we know that $\Omega\sigma^\ast\loc X\cong\Omega\loc\sigma^\ast X$. On $\Sp\D$, $\Omega$ is a continuous morphism, so it commutes with right Kan extensions as well as pullback morphisms. Therefore the only sticking point is whether $\Omega$ commutes with the sequential colimit $\pi_!\colon \Sp\D^\N\to\Sp\D$.

Recall from Definition~\ref{defn:suspensionloop} how $\Omega$ is constructed. It is the composition
\begin{equation*}
\xymatrix{
\Sp\D\ar[r]^-{(1,1)_!}&\Sp\D^\ur\ar[r]^-{i_{\ur,\ast}}&\Sp\D^\square\ar[r]^-{(0,0)^\ast}&\Sp\D.
}
\end{equation*}
Investigating this, truly the only sticking point is commuting with the middle morphism. But since we have assumed $\D$ is regular, $\Sp\D$ is also regular by Proposition~\ref{prop:localizeregular}. If the sequential (hence filtered) colimit $\pi_!$ commutes with finite limits, it also commutes with $i_{\ur,\ast}$ using Der4. Hence $\phi_{\loc X}$ is an isomorphism and $\loc X\in\St\D$.

For $Y\in\St\D$, $\loc Y\cong~Y$ as each of the morphisms $\Omega^{n}(\sigma^\ast)^n Y\to\Omega^{n+1}(\sigma^\ast)^{n+1}Y$ in $F(Y)$ is an isomorphism, and the colimit of a sequence of isomorphisms is again an isomorphism. Therefore the essential image of $\loc$ is all of $\St\D$.
\end{proof}

We have one more lemma to conclude the theorem.

\begin{lemma}\label{lemma:locisaloc}
There exists a modification $\eta\colon\id_{\Sp\D}\to\loc$ such that $\loc$ is a localization morphism.
\end{lemma}
\begin{proof}
We construct $\eta\colon\id_{\Sp\D}\to\loc$ such that $\eta_{\loc}$ and $\loc\eta$ are isomodifications. Applying Proposition~\ref{prop:locofders}, we will obtain a localization with essential image $\St\D$, so the inclusion is right adjoint to $\loc$.

First, consider the category $A$ given by the following diagram:
\begin{equation*}
\xymatrix@C=1em{
\cdots \ar[r]<0.5ex>& 2'\ar[r]<0.5ex>\ar[l]<0.5ex> & 1'\ar[r]<0.5ex>\ar[l]<0.5ex> & 0'\ar[r]\ar[l]<0.5ex> & 0\ar[r] & 1\ar[r] & 2\ar[r] & \cdots
}
\end{equation*}
where the composition of any two stacked morphisms is the identity (in either direction). There is a natural projection $p\colon A\to\N$ onto the unmarked objects, where $p(n')=0$ by necessity. If we have an object $Y=(Y_0\overset{f_0}\longrightarrow Y_1\overset{f_1}\longrightarrow\cdots)\in\E(\N)$ in some (pre)derivator $\E$, then $p^\ast Y$ has the form
\begin{equation*}
\xymatrix{
\cdots \ar[r]<0.5ex>^=& Y_0\ar[r]<0.5ex>^=\ar[l]<0.5ex>^=& Y_0\ar[r]<0.5ex>^=\ar[l]<0.5ex>^= & Y_0\ar[r]^=\ar[r]<-0.1ex>\ar[l]<0.5ex>^= & Y_0\ar[r]^{f_0} & Y_1\ar[r]^{f_1} & Y_2\ar[r] & \cdots
}
\end{equation*}

We will fold this diagram along the morphism $0'\to 0$, bolded above. Note that there is exactly one morphism $n'\to n$ for every $n$, and so we may rewrite our diagram with this in mind.
\begin{equation*}
\vcenter{\xymatrix{
Y_0\ar[r]^{f_0}&Y_1\ar[r]^{f_1}&Y_2\ar[r]&\cdots\\
Y_0\ar[u]^=\ar[r]<0.5ex>^=&Y_0\ar[u]^{f_0}\ar[r]<0.5ex>^=\ar[l]<0.5ex>^=&Y_0\ar[u]^{f_1f_0}\ar[r]<0.5ex>^=\ar[l]<0.5ex>^=&\cdots\ar[l]<0.5ex>^=
}}
\end{equation*}
If we now forget the leftward arrows on the bottom row, we obtain the category $\N\times[1]$. This gives us in total a functor $a\colon \N\times[1]\to A\to \N$. Now, recall that $\loc = \pi_! F$ and let $X\in\Sp\D$. Then $F(X)\in\Sp\D^\N$, and so we may apply $a^\ast$ to obtain an object in $\Sp\D^{\N\times[1]}$, \ie a coherent morphism of $\N$-shaped diagrams. To examine this more closely, recall that $F(X)$ has underlying diagram
\begin{equation*}
\xymatrix{
X\ar[r]^-{f_0}&\Omega\sigma^\ast X\ar[r]^-{f_1}&\Omega^2(\sigma^\ast)^2 X\ar[r]&\cdots
}
\end{equation*}
Then $a^\ast F(X)$ has underlying diagram
\begin{equation}\label{dia:coherentunit}
\vcenter{\xymatrix{
X\ar[r]^-{f_0}&\Omega\sigma^\ast X\ar[r]^-{f_1}&\Omega^2(\sigma^\ast)^2 X\ar[r]&\cdots\\
X\ar[u]^=\ar[r]^=&X\ar[u]^{f_0}\ar[r]^=&X\ar[u]^{f_1f_0}\ar[r]^=&\cdots
}}
\end{equation}
Applying $\pi_!$ gives rise to a morphism of derivators $\mu=\pi_!a^\ast F\colon \Sp\D\to \Sp\D^{[1]}$, which we can view as a coherent modification. The underlying actual modification is $\dia\mu\colon s^\ast\mu\Rightarrow~t^\ast\mu$, where $s,t\colon e\to[1]$ are the source and target functors.

The functor $\pi\colon \N\to e$ admits a left adjoint, namely $0\colon e\to \N$ which classifies the initial object. By \cite[Proposition~1.18]{Gro13}, the canonical modification\linebreak $\varepsilon\colon\pi_!\pi^\ast\to\id_\E$ is an isomodification for any (left) derivator $\E$. In our case, this shows that $s^\ast\mu=\pi_!\pi^\ast\cong~\id_{\Sp\D}$. By construction $t^\ast\mu=\loc$, so we obtain a a modification $\eta\colon\id_{\Sp\D}\to \loc$ by precomposing $\dia\mu$ with the isomodification $\varepsilon\inv\colon\id_{\Sp\D}\to\pi_!\pi^\ast$.

Let $X\in\Sp\D$. We now need to show that $\eta_{\loc X}$ and $\loc\eta_X$ are both isomorphisms. For the first, because $\loc X$ is a stable spectrum, each of the maps $f_i$ in Diagram~\ref{dia:coherentunit} are isomorphisms. Therefore each of the vertical maps of $a^\ast F(X)$ are isomorphisms, and the colimit of isomorphisms is again an isomorphism. Therefore $\eta_{\loc X}$ is an isomorphism.

Now, we would like to compare the formulae for $\eta_{\loc}$ and $\loc\eta$. Specifically, we have a diagram
\begin{equation}\label{dia:muloc}
\vcenter{\xymatrix@R=1em{
&&\Sp\D\ar[r]^-F&\Sp\D^\N\ar@(r,ul)[dr]^-{a^\ast}\ar@{}[ddl]|\swtrans&&\\
\Sp\D\ar[r]^-F&\Sp\D^\N\ar@(ur,l)[ur]^-{\pi_!}\ar@(dr,l)[dr]_-{a^\ast}&&&\Sp\D^{\N\times[1]}\ar[r]^-{\pi_!^{[1]}}&\Sp\D^{[1]},\\
&&\Sp\D^{\N\times[1]}\ar[r]_-{\pi_!}&\Sp\D^{[1]}\ar@(r,dl)[ur]_-{F^{[1]}}&&
}}
\end{equation}
where the top composition is $\mu\loc$ and the bottom is $\loc^{[1]}\mu$. The middle transformation encodes the canonical isomorphism given by the zigzag
\begin{equation*}
\vcenter{\xymatrix{
a^\ast F\pi_! Y\ar[r]^-\cong &F^{[1]}a^\ast\pi_! Y& F^{[1]}\pi_!a^\ast Y\ar[l]_-\cong
}}
\end{equation*}
for any $Y\in\Sp\D^\N$ (as $a^\ast$ is bicontinuous). We obtain $\eta_{\loc}$ by precomposing $\dia(\mu\loc)$ by $\varepsilon\inv_{\loc}$ and $\loc\eta$ by precomposing $\dia(\loc^{[1]}\mu)$ by $\loc\varepsilon\inv$. Because $\varepsilon\inv_{\loc}$ and $\eta_{\loc}$ are isomorphisms, this means that $\mu\loc X$ is a coherent isomorphism for any $X\in\Sp\D$. But this means that $\loc^{[1]}\mu X$ is a coherent isomorphism as well by Diagram~\ref{dia:muloc}, so the composition $\loc\varepsilon\inv\dia(\loc^{[1]}\mu) X=\loc\eta_X$ is an isomorphism for any $X\in\Sp\D$.

Thus we have established that Proposition~\ref{prop:locofders} applies, and so complete the proof of Theorem~\ref{thm:stdlocalisation}.
\end{proof}

To complete our goal for this section, we need one more proposition.

\begin{prop}
$\St\D$ is a stable derivator.
\end{prop}
\begin{proof}
Lemma~\ref{lemma:localisederivator} tells us that $\St\D$ is a derivator. It is also clear that $0\in\Sp\D(e)$ satisfies $0\cong \Omega\sigma^\ast 0$, so $0\in\St\D(e)$ and $\St\D$ is pointed.

To show that $\St\D$ is stable, we can prove one of many equivalent formulations in \cite[Theorem~3.1]{Gro16b}. The easiest one for us would be to prove that $(\Sigma,\Omega)$ is an adjoint equivalence on $\St\D$.

Recall that $\id_{\St\D}\to\Omega\sigma^\ast$ is an invertible modification of endomorphisms of $\St\D$. Moreover, $\sigma^\ast$ is an automorphism of $\St\D$. The map of diagrams $\sigma\colon V\to V$ is invertible, namely by $(i,j)\mapsto (j-1,i-1)$. Precomposition with the modification above gives us $(\sigma^{-1})^\ast=(\sigma^\ast)^{-1}\to \Omega$ is an isomodification, hence $\Omega$ is an automorphism as well (and a fortiori an equivalence). This forces $\Sigma$ to be an equivalence as well.
\end{proof}
\begin{cor}\label{cor:shiftsuspension}
$\sigma^{\ast}$ and $\Sigma$ are canonically isomorphic in $\St\D$.
\end{cor}
In general any two adjoints to $\Omega$ must be canonically isomorphic, and the above shows that these are indeed two choices of left adjoint.

All in all we have two adjunctions
\begin{equation*}
\vcenter{\xymatrix{
\D\ar@/_1pc/[d]_L \\
\Sp\D\ar@/_1pc/[d]_{\loc}\ar@/_1pc/[u]_{(0,0)^\ast} \\
\St\D\ar@/_1pc/[u]_i
}}
\end{equation*}

\begin{defn}\label{defn:stabdefn}
Let $\D$ be a regular pointed derivator. The \emph{stabilization of $\D$} is the left adjoint morphism $\stab=\loc L\colon \D\to\St\D$ constructed above. Its right adjoint is $(0,0)^\ast i\colon\St\D\to\D$.
\end{defn}

\begin{remark}
While $\D\to\St\D$ is not a localization, the map $\D^V\to\Sp\D\to\St\D$ is a composition of localizations, hence is itself a localization. Therefore we can think of the stabilization of a derivator in terms of a localization of a certain diagram category on that derivator.

Alternatively, there is another model of $\St\D$ as $\Sigma$-spectra and a colocalization $\Sp\D\to\St\D$ onto the $\Sigma$-spectra. This would give us a different stabilization morphism $\stab_\Sigma\colon\D\to\St\D$ for \emph{coregular} pointed derivators $\D$ which would be right adjoint to $(0,0)^\ast$. But because regularity seems a more natural condition (as far as examples go), we have chosen to prove everything through the lens of $\Omega$-spectra.
\end{remark}

\begin{remark}
Regularity may not be necessary for stabilization, as demonstrated in personal correspondence with Cisinski. Consider a combinatorial model category, which arise (up to Quillen equivalence) as Bousfield localizations of model categories of simplicial presheaves with pointwise weak equivalences by Dugger's presentation theorem \cite[Theorem~1.1]{Dug01}. Such categories of simplicial presheaves give rise to regular derivators, and work of Cisinski and Tabuada in \cite{Tab08} proves that left Bousfield localizations of model categories give rise to left Bousfield localizations of their corresponding derivators.

There is model-categorical machinery giving every combinatorial model category a canonical stabilization, which then passes to a stabilization (in some sense) on the derivator side of things as well. However, there is no reason to believe than an arbitrary combinatorial model category is still regular; Bousfield localizations of regular derivators need not be regular any longer. Therefore it is plausible that, in general, Bousfield localizations of regular derivators may still admit a canonical stabilization in the sense of this paper.
\end{remark}
\ \newline

\section{The universal property of stabilization}

We begin with a proposition to confirm that we have not developed an aberrant stabilization theory.

\begin{prop}\label{prop:stabonstable}
If $\bS$ is a stable derivator, then $\stab\colon\bS\to\St\bS$ is an equivalence of derivators.
\end{prop}
\begin{proof}
First, a stable derivator is both pointed and regular, so this statement makes sense. Regularity follows by~\cite[Theorem~3.15(iii.b)]{Gro16b}. As we noted above, in a stable derivator, (homotopy) finite limits (such as pullbacks) commute with arbitrary colimits.

We would like to check that the unit and counit of $(\stab,(0,0)^\ast i)$ are isomodifications. To that end, let $\eta\colon\id_\bS\to (0,0)^\ast\stab$ be the unit and $\varepsilon\colon\stab(0,0)^\ast i\to \id_{\St\bS}$ the counit. Let $x\in \bS$, so that the unit evaluates to
\begin{equation*}
\eta_x\colon x\to (0,0)^\ast i\loc(Lx),
\end{equation*}
where $L\colon\bS\to\Sp\bS$ is the morphism taking $x$ to its connective suspension spectrum. This map is the composition of the units $\eta_1\colon \id_{\bS}\to (0,0)^\ast L$ and $\eta_2\colon \id_{\Sp\bS}\to i\loc$:
\begin{equation*}
\vcenter{\xymatrix@C=2em{
x\ar[r]^-{\eta_{1,x}}&(0,0)^\ast Lx\ar[rr]^-{(0,0)^\ast\eta_{2,Lx}}&&(0,0)^\ast i\loc(Lx).
}}
\end{equation*}

The first unit $\eta_1$ is an isomodification because $L$ is fully faithful. Recall that $\eta_2$ is a composition of an isomodification with a modification $\pi_!\pi^\ast\Rightarrow~\pi_! F$, so let us examine first the map $\pi^\ast Lx\to F(Lx)$. We know that $F(Lx)$ has the form
\begin{equation*}
\vcenter{\xymatrix{
Lx\ar[r]&\Omega\sigma^\ast Lx\ar[r]&\Omega^2(\sigma^\ast)^2Lx\ar[r]&\cdots
}}
\end{equation*}

Now, rather than compute $(0,0)^\ast \pi_!\pi^\ast Lx\to(0,0)^\ast \pi_! F(Lx)$ (where we suppress the inclusion $i$), we can check the map $(0,0)^\ast \pi^\ast Lx\to~ (0,0)^\ast F(Lx)$ and take $\pi_!$ of this map. Because $Lx$ is a suspension spectrum, $(0,0)^\ast(\sigma^\ast)^nLx=(n,n)^\ast Lx\cong\Sigma^nx$. Therefore
\begin{equation*}
(0,0)^\ast F(Lx)=(x\to \Omega\Sigma x\to \Omega^2\Sigma^2 x\to\cdots)\in\bS^\N
\end{equation*}
and each map is an isomorphism because $\bS$ is stable. Hence the map in question
\begin{equation*}
\vcenter{\xymatrix{
x\ar[r]^-\cong&\Omega\Sigma x\ar[r]^-\cong&\Omega^2\Sigma^2 x\ar[r]^-\cong&\cdots\\
x\ar[r]_-=\ar[u]^-=&x\ar[r]_-=\ar[u]&x\ar[r]_-=\ar[u]&\cdots
}}
\end{equation*}
is an isomorphism at every $n\in\N$, so that $\pi_!$ of this map is still an isomorphism. Therefore $(0,0)^\ast\eta_{2,Lx}$ is an isomorphism and thus $\eta_x$ is as well.

To examine the counit, let $Y\in\St\bS$. Then the counit is
\begin{equation*}
\varepsilon_Y\colon\loc (L(0,0)^\ast iY)\to Y.
\end{equation*}
For simplicity, write $Y_0=(0,0)^\ast iY$ as usual. The counit is the composition of\linebreak $\varepsilon_1\colon L(0,0)^\ast\to\id_{\Sp\bS}$ and $\varepsilon_2\colon \loc i\to \id_{\St\bS}$:
\begin{equation*}
\vcenter{\xymatrix@C=2em{
\loc (LY_0)\ar[rr]^-{\loc\varepsilon_{1,iY}}&&\loc iY\ar[r]^-{\varepsilon_{2,Y}}&Y
}}
\end{equation*}
The second counit $\varepsilon_2$ is an isomodification, so we need only to understand the first counit. In particular, we need to show that it is an isomorphism under $\loc$.

$LY_0$ is the connective suspension spectrum on $Y_0$, which is isomorphic to the truncation of the original $Y$ at the 0th level. The map $LY_0\to Y$ encodes the (adjoints of the) canonical comparison isomorphisms $\Sigma^n Y_0\to Y_n$ for $n\geq 0$ and is the zero map otherwise. We now need to compute $\loc$ applied to this map, which is $\pi_!$ applied to $F(LY_0)\to\pi^\ast Y$ followed by the canonical isomorphism $\pi_!\pi^\ast Y\to Y$.

We now examine the coherent diagram $F(LY_0)$. The spectrum $\Omega\sigma^\ast LY_0$ is given by
\begin{equation*}
(i,i)^\ast\Omega\sigma^\ast LY_0\cong\Omega \Sigma^{i+1}Y_0\cong\Sigma^{i}Y_0
\end{equation*}
for any $i\geq -1$ and zero elsewhere. In general, $\Omega^n(\sigma^\ast)^nLY_0$ is isomorphic to the truncation of the original spectrum $Y$ at the $-n$th level. We can compute the first few entries of the underlying diagram of $F(LY_0)\in\bS(V\times\N)$, letting $\rho\colon \id_{\bS}\to\Omega\Sigma$ be the unit of the $(\Sigma,\Omega)$ equivalence on $\bS$.
\begin{equation}\label{dia:stabonstable}
\vcenter{\xymatrix@R=0.5em@C=3em{
i&\vdots &\vdots &\vdots &\\
-2&0\ar[r] &0\ar[r] &\Omega^2 Y_0\ar[r] &\cdots\\
-1&0\ar[r] &\Omega Y_0\ar[r]^-{\Omega\rho_{Y_0}} &\Omega^2\Sigma Y_0 \ar[r]&\cdots\\
0&Y_0\ar[r]^-{\rho_{Y_0}} &\Omega\Sigma Y_0\ar[r]^-{\Omega\rho_{\Sigma Y_0}} &\Omega^2\Sigma^2 Y_0\ar[r] &\cdots\\
1&\Sigma Y_0\ar[r]^-{\rho_{\Sigma Y_0}}  &\Omega\Sigma^2 Y_0 \ar[r]^-{\Omega\rho_{\Sigma^2Y_0}} & \Omega^2\Sigma^3 Y_0\ar[r] &\cdots\\
&\vdots&\vdots&\vdots&
 }}
\end{equation}
The vertical chains represent $\Omega^n(\sigma^\ast)^n LY_0$ (from left to right), where we write only the $(i,i)$ entries and mark the value of $i$ on the lefthand side. The colimit $\pi_! F(LY_0)$ is computed using the horizontal morphisms, which we see are eventually isomorphisms for any $i\in\Z$.

We will now prove that $\pi_! F(LY_0)\to \pi_!\pi^\ast Y$ is an isomorphism by showing that it is a pointwise isomorphism. We may do this by computing $(i,i)^\ast$ applied to the above map $F(LY_0)\to \pi^\ast Y$ and then applying $\pi_!$, as $\pi_!$ commutes with $(i,i)^\ast$ up to canonical isomorphism. For $i\geq 0$, this map is
\begin{equation*}
\vcenter{\xymatrix@C=2em{
\Sigma^iY_0\ar[r]^-{\cong} &\Omega\Sigma^{i+1} Y_0\ar[r]^-{\cong} &\Omega^2\Sigma^{i+2} Y_0\ar[r]^-\cong &\cdots\\
Y_i\ar[r]^-=\ar@{<-}[u]^-\cong&Y_i\ar[r]^-=\ar@{<-}[u]&Y_i\ar[r]^-=\ar@{<-}[u]&\cdots
}}
\end{equation*}
which we see consists solely of isomorphisms, so $\pi_!$ of this map is an isomorphism. For $i<0$, we have a leading trail of zeroes:
\begin{equation*}
\vcenter{\xymatrix@C=2em{
0\ar[r]&\cdots\ar[r]&0\ar[r]&\Omega^{-i}Y_0\ar[r]^-{\cong} &\Omega^{-i}\Sigma Y_0\ar[r]^-{\cong} &\Omega^{-i}\Sigma^2 Y_0\ar[r]^-\cong &\cdots\\
Y_i\ar[r]^-=\ar@{<-}[u]&\cdots\ar[r]&Y_i\ar[r]^-=\ar@{<-}[u]&Y_i\ar[r]^-=\ar@{<-}[u]^-\cong&Y_i\ar[r]^-=\ar@{<-}[u]&Y_i\ar[r]^-=\ar@{<-}[u]&\cdots
}}
\end{equation*}
These maps are not all isomorphisms, but they are eventually isomorphisms. Therefore $\pi_!$ applied to this map is an isomorphism as well, which shows that $\loc(\varepsilon_{1,Y})$ is an isomorphism by Der2.

Having shown that both the unit and the counit of the adjunction are isomodifications, this proves the equivalence.
\end{proof}

We would now like to prove a universal property of stabilization of the following form: let $\bS$ be a stable derivator, and let $\D$ be a regular pointed derivator. Then we have an equivalence of categories
\begin{equation*}
\stab^\ast\colon\Hom_\bullet(\St\D,\bS)\to\Hom_\bullet(\D,\bS)
\end{equation*}
given by precomposition with $\stab\colon\D\to\St\D$, where $\Hom_\bullet\subset\Hom$ is the full subcategory on a certain class of morphisms. It would be surprising if $\Hom_\bullet$ were everything, as different morphisms $\St\D\to\bS$ which do not involve the stability of the domain and the codomain should not necessarily compose with $\stab$ to different morphisms $\D\to\St\D\to\bS$. We take as inspiration~\cite[Theorem~8.1]{Hel97} that $\bullet=\;!$, the subcategory of cocontinuous morphisms.

We deviate from Heller's proof, however, for two reasons. First, Heller's proof relies on $\D$ being a strong derivator. If $\D$ is strong, $\St\D$ is a strong, stable derivator so $\St\D(e)$ admits a canonical triangulated structure; see \cite[\S3]{Mal07} or \cite[\S4.2]{Gro13}. A study of `stable equivalences' and `stably trivial objects' gives him a kind of Verdier quotient from $\Sp\D\to\St\D$. But our derivators are not assumed strong, and stability alone is not sufficient to guarantee a triangulation.

The second and more important reason is that Heller's proof lacks key details. In particular, consider \cite[Diagram~9.3]{Hel97}, an infinite diagram which he constructs incoherently. Despite following his reference to \cite[III~\S3]{Hel88}, Heller does not prove that a diagram of shape $U$ (in his terminology) will lift coherently even if we assume that $\dia_{[n]}$ is full and essentially surjective for any $n\in\N$. Moreover, modern derivator literature has produced no (infinite) lifting results of the sort Heller would require, despite his hope that strong derivators have $\dia_K$ full and essentially surjective `for a much larger class of categories' $K$ than free ones \cite[III~\S3]{Hel88}. Because we have not assumed our derivators strong, we would not have access to such results in any case.

\begin{defn}
Let $\Phi\colon \D\to \E$ be a pointed morphism of regular pointed derivators. Then define $\St \Phi\colon \St\D\to\St\E$ by the composition
\begin{equation*}
\xymatrix{
\St\D\,\ar[r]^{i_\D}&\Sp\D\ar[r]^-{\Phi^V}&\Sp\E\ar[r]^-{\loc_\E}&\St\E.
}
\end{equation*}
We require that $\Phi$ be pointed in order that $\Phi^V\colon\Sp\D\to\E^V$ have its image in $\Sp\E$.
\end{defn}

\begin{lemma}\label{lemma:stableequivalence}
Let $\Phi\colon \D\to\E$ be a cocontinuous morphism of regular pointed derivators. Then the following square commutes up to invertible modification:
\begin{equation*}
\xymatrix{
\Sp\D\ar[d]_-{\loc_\D}\ar[r]^{\Phi^V}&\Sp\E\ar[d]^{\loc_\E}
\ar@{}[dl]|\swtrans \ar@{}[dl]<0.4ex>_\cong\\
\St\D\ar[r]_-{\St \Phi}&\St\E
}
\end{equation*}
\end{lemma}
\begin{proof}
This modification is given by the unit of the localization $\eta_\D\colon \id_{\Sp\D}\to i_\D \loc_\D$. This gives us a map
\begin{equation}\label{eqn:hardmodification}
\xymatrix@C=4em{
\loc_\E \Phi^V\ar[r]^-{\loc_\E \Phi^V\eta_\D}&\loc_\E \Phi^V i_\D\loc_\D=\St \Phi\loc_\D\!.
}
\end{equation}
We just need to prove that this is an isomodification. We will prove that, for any $X\in\Sp\D$, $\loc_\E\Phi^V\eta_{\D,X}$ is an isomorphism explicitly, though this will involve changing its domain and codomain up to isomorphism.

As we will be working with many different diagram categories shortly, we will also write $\Phi$ for the appropriate $\Phi^K$. Remembering $\loc=\pi_! F$, we know that $\Phi\pi_!\cong\pi_!\Phi$ because $\Phi$ is cocontinuous, so our codomain is isomorphic to $\pi_! F \pi_! \Phi F(X)$. We claim that we may commute the middle $F$ and $\pi_!$ by regularity.

Let $Y\in\Sp\E^\N$ and examine the comparison map $\pi_! F(Y)\to F\pi_!(Y)$. The underlying diagram is
\begin{equation*}
\vcenter{\xymatrix{
\pi_!Y\ar[r]\ar[d]& \pi_!\Omega\sigma^\ast Y\ar[r]\ar[d]&\pi_!\Omega^2(\sigma^\ast)^2 Y\ar[r]\ar[d]&\cdots\\
\pi_!Y\ar[r]&\Omega\sigma^\ast \pi_!Y\ar[r]&\Omega^2(\sigma^\ast)^2 \pi_!Y\ar[r]&\cdots
}}
\end{equation*}
Hence $\pi_!$ commutes with $F$ so long as $\Omega$ (and $\sigma^\ast$) commutes with $\pi_!$. Since we have assumed $\E$ is regular, this is so.

Therefore we may compute $\loc_\E\Phi^V i_\D \loc_\D(X)$ as $\pi_!\pi_! F\Phi F(X)$, that is, the colimit of an $\N^2$-shaped diagram. Let us adopt the shorthand $X^i=\Omega^i(\sigma^\ast)^i X$. Then the underlying diagram of $F\Phi F(X)$ is
\begin{equation}\label{dia:waffle}
\vcenter{\xymatrix@C=1em@R=1em{
\Phi X^0\ar[r]\ar[d]& \Phi X^1\ar[d]\ar[r]&\Phi X^2\ar[r]\ar[d]&\cdots\\
\Omega\sigma^\ast\Phi X^0\ar[r]\ar[d]& \Omega\sigma^\ast\Phi X^1\ar[d]\ar[r]&\Omega\sigma^\ast\Phi X^2\ar[r]\ar[d]&\cdots\\
\Omega^2(\sigma^\ast)^2\Phi X^0\ar[r]\ar[d]& \Omega^2(\sigma^\ast)^2\Phi X^1\ar[d]\ar[r]&\Omega^2(\sigma^\ast)^2\Phi X^2\ar[r]\ar[d]&\cdots\\
\vdots&\vdots&\vdots
}}
\end{equation}
In short, $(p,q)^\ast F\Phi F(X)= \Omega^q(\sigma^\ast)^q\Phi(X^p)$. Note that the lefthand column ($p=0$) is $F\Phi^V (X)$, the domain of our missing map before applying $\pi_!$. Indeed, the unit $\loc_\E\Phi^V\eta_{\D,X}$ is the colimit over $\N^2$ of the map from $\pi^\ast$ applied to the first column to the diagram as a whole. For a restricted portion of $\N^2$, this map is
\begin{equation*}
\vcenter{\xymatrix@C=1em@R=1em{
\Phi X^0\ar[r]^-=\ar[d]& \Phi X^0\ar[d]\ar[r]^-=&\Phi X^0\ar[d]\\
\Omega\sigma^\ast\Phi X^0\ar[r]^-=& \Omega\sigma^\ast\Phi X^0\ar[r]^-=&\Omega\sigma^\ast\Phi X^0\\
}}
\longrightarrow
\vcenter{\xymatrix@C=1em@R=1em{
\Phi X^0\ar[r]\ar[d]& \Phi X^1\ar[d]\ar[r]&\Phi X^2\ar[d]\\
\Omega\sigma^\ast\Phi X^0\ar[r]& \Omega\sigma^\ast\Phi X^1\ar[r]&\Omega\sigma^\ast\Phi X^2\\
}}
\end{equation*}

We will show that $\loc_\E\Phi^V\eta_{\D,X}$ is an isomorphism directly by decomposing it into the composition of two canonical colimit comparison maps which we prove are isomorphisms, namely Equations~\ref{eq:map1}~and~\ref{eq:map2}.

In order to do this we show that Diagram~\ref{dia:waffle} is a final subdiagram of a diagram which is essentially \cite[Diagram~9.3]{Hel97}. In the theme of this paper, Heller's argument is upgradeable to a coherent version. The following argument is an extra-dimensional enhancement of Construction~\ref{cons:coherentdia}.

Recall from Equation~\ref{eq:locdefinition} that the functor $F\colon\Sp\D\to\Sp\D^\N$ had as its last steps a right Kan extension and a restriction to $\{(0,0)\}\times \N\subset V^{\geq 0}\times \N$. Let us denote by $G=f_\ast\zeta^\ast\omega^\ast\colon\Sp\D\to\Sp\D^{V^{\geq 0}\times \N}$ all but the ultimate step of that construction. For an object $X\in\Sp\D$, $G(X)$ has the following form:
\begin{equation*}
\vcenter{\xymatrix@R=1ex@C=1ex{
& & & & 0\ar[rr]\ar'[d][dd] & & (\sigma^\ast)^2 X\ar[dd]\\
& 0\ar[rr]\ar'[d][dd] & & \sigma^\ast X\ar[dd]\ar[rr]\ar[ur] & & 0\ar[ur]\ar[dd]\\
X\ar[dd]\ar[ur]\ar[rr] & & 0\ar[ur]\ar[dd] & & 0\ar'[r][rr]\ar'[d][dd] & & (\sigma^\ast)^2 X\ar[dd]\\
& 0\ar'[r][rr]\ar'[d][dd] & & \sigma^\ast X\ar[dd]\ar[ur]\ar[rr] & & 0\ar[ur]\ar[dd]\\
\Omega\sigma^\ast\ar[dd]\ar[ur]\ar[rr] X& & 0\ar[ur]\ar[dd] & & 0\ar'[r][rr]\ar@.[d] & & (\sigma^\ast)^2 X\ar@.[d]\\
& 0\ar'[r][rr]\ar@.[d] & & \Omega(\sigma^\ast)^2 X\ar[rr]\ar[ur]\ar@.[d] & & 0\ar[ur]\ar@.[d]& \\
\Omega^2(\sigma^\ast)^2 X\ar[ur]\ar[rr]\ar@.[d]& & 0\ar[ru]\ar@.[d] & & & &\\
& & & & & & &
}}
\end{equation*}

We now apply the morphism $\Phi$ to the entire diagram. The $(0,0)$ column is now of the form $\Phi (X^i)$, so it is the first row of Diagram~\ref{dia:waffle}. We now need to construct coherently the rest of Diagram~\ref{dia:waffle}, and to do so we will need to add yet another dimension to this diagram.

Consider the following subcategory of $V^{\geq 0}\times \N_1\times \N_2$, where we imagine $\Phi G(X)$ embedding into the first two coordinates and aiming for a constant diagram in the $\N_2$-direction. At each $k\in \N_1$, we restrict to $\{0,\ldots, i\}\subset \N_2$ for $(i,i)\in V^{\geq 0}$ with $i<k$ and $\{0,\ldots,k\}\subset \N_2$ otherwise. Call this shape $B_1$ and let $b_1\colon B_1\to V^{\geq 0}\times \N$ be the projection forgetting the last coordinate. Then $b_1^\ast \Phi G(X)$ is equal to $\Phi G(X)$ when the last coordinate is zero, and constant of variable lengths in the $\N_2$-direction.

For example, at $k=2$ we have the subcategory of $V^{\geq 0}\times\{0,1,2\}$:
\begin{equation}\label{dia:extradimension}
\vcenter{\xymatrix@R=1ex@C=1ex{
& & & & & & & {} \\
& & & & 0\ar[rr]\ar'[d][dd] & & \Phi(\sigma^\ast)^2 X\ar[dd]\ar@.[ur]\\
& 0\ar[rr]\ar'[d][dd] & & \Phi\Omega(\sigma^\ast)^2 X\ar[dd]\ar[rr]\ar[ur] & & 0\ar[ur]\ar[dd]&&{}\\
\Phi\Omega^2(\sigma^\ast)^2X\ar[ur]\ar[rr] & & 0\ar[ur]\ar[dd] & & 0\ar'[r][rr]\ar'[d][dd] & & \Phi(\sigma^\ast)^2 X\ar[dd]\ar@.[ur]\\
& 0\ar'[r][rr]\ar[dd] & & \Phi\Omega(\sigma^\ast)^2 X\ar[ur]\ar[rr] & & 0\ar[ur]\ar[dd]&&{}\\
& & 0\ar[ur]\ar[dd] & & 0\ar'[r][rr] & & \Phi(\sigma^\ast)^2 X\ar@.[ur]\\
& 0\ar[rrru]|!{[ur];[dr]}\hole & & & & 0\ar[ur]\\
& & 0\ar[rrru] & & & &
}}
\end{equation}
continuing in a constant vertical diagram of shape $[2]$ in the $V^{\geq 0}$ dimension but not downward in the $\N_2$ dimension. This is an analogous construction to Diagram~\ref{dia:Z}, but depending on both $(i,j)\in V^{\geq 0}$ and $k\in\N_1$.

Let $B_2$ be the subdiagram $V^{\geq 0}\times \N_1\times \N_2$ which has at $k\in \N_1$ the entire $\{0,\ldots ,k\}~\subset \N_2$ regardless of $(i,j)\in V^{\geq 0}$, and let $b_2\colon B_1\to B_2$ be the inclusion. Then we would like to compute $b_{2,\ast}b_1^\ast \Phi G(X)$. Specifically, we would like to investigate the squares
\begin{equation*}
\xymatrix{
(i,i+1,k,l)\ar[r]&(i+1,i+1,k,l)\\
(i,i,k,l)\ar[u]\ar[r]&(i+1,i,k,l)\ar[u]
}
\end{equation*}
for $i< l\leq k$. The lower left corner does not lie in the image of $b_2$, so we will again use Proposition~\ref{prop:cartdetectlemma} to determine that these squares are cartesian. Though we have another dimension, we are still working in a poset, so the slice categories we construct will be subcategories of $B_2$.

Rather than attempt to describe $(B_2\setminus(i,i,k,l))_{(i,i,k,l)/}$, let us begin immediately to construct the right adjoint $r$ to $\ur\to (B_2\setminus(i,i,k,l))_{(i,i,k,l)/}$. The adjoint is nearly the same as the one we constructed in our first application of Proposition~\ref{prop:cartdetectlemma} in the Construction~\ref{cons:coherentdia}. We use the same notation for $\ur$ as we did there, and we temporarily abandon entirely our conventions, having run out of plausible variable names:
\begin{equation*}
r(w,x,y,z)=\begin{cases}
(1,1)&w,x\geq i+1\\ 
(1,0)&(w,x)=(i+1,i)\\
(0,1)&(w,x)=(i,i+1)
\end{cases}
\end{equation*}
As in the earlier argument, the $\N$-dimensions do not matter, because any \linebreak$(w,x,y,z)\in B_2$ which is in the subcategory $(B_2\setminus(i,i,k,l))_{(i,i,k,l)/}$ will always satisfy $(y,z)\geq (k,l)$. That this is a right adjoint can be verified identically to above.

This proves that all such squares are cartesian. Aiming again to calculate the values of our diagram at $(0,0,k,l)$, we can first look at the square with lower left corner $(l-1,l-1,k,l)$. The rest of this square is entirely in $B_1$, with underlying diagram
\begin{equation*}
\xymatrix@C=1em@R=1em{
0\ar[r]&\Phi\Omega^{k-l}(\sigma^\ast)^k X\\
&0\ar[u]
}
\end{equation*}
This makes sense because we necessarily have $l\leq k$. After applying $b_{2,\ast}$ we obtain a cartesian square, hence
\begin{equation*}
(l-1,l-1,k,l)^\ast b_{2,\ast}b_1^\ast \Phi G(X)\cong \Omega\Phi\Omega^{k-l}(\sigma^\ast)^k X
\end{equation*}
and so by induction
\begin{equation*}
(0,0,k,l)^\ast b_{2,\ast}b_1^\ast \Phi G(X)\cong \Omega^l\Phi\Omega^{k-l}(\sigma^\ast)^k X.
\end{equation*}
To give the result of Diagram~\ref{dia:extradimension} after $b_{2,\ast}$:
\begin{equation}\label{dia:extradimensionafter}
\vcenter{\xymatrix@R=1ex@C=1ex{
& & & & & & & {} \\
& & & & 0\ar[rr]\ar'[d][dd] & & \Phi(\sigma^\ast)^2 X\ar[dd]\ar@.[ur]\\
& 0\ar[rr]\ar'[d][dd] & & \Phi\Omega(\sigma^\ast)^2 X\ar[dd]\ar[rr]\ar[ur] & & 0\ar[ur]\ar[dd]&&{}\\
\Phi\Omega^2(\sigma^\ast)^2X\ar[dd]\ar[ur]\ar[rr] & & 0\ar[ur]\ar[dd] & & 0\ar'[r][rr]\ar'[d][dd] & & \Phi(\sigma^\ast)^2 X\ar[dd]\ar@.[ur]\\
& 0\ar'[r][rr]\ar'[d][dd] & & \Phi\Omega(\sigma^\ast)^2 X\ar[dd]\ar[ur]\ar[rr] & & 0\ar[ur]\ar[dd]&&{}\\
\Omega\Phi\Omega(\sigma^\ast)^2X\ar[dd]\ar[ur]\ar[rr]& & 0\ar[ur]\ar[dd] & & 0\ar'[r][rr] & & \Phi(\sigma^\ast)^2 X\ar@.[ur]\\
& 0\ar'[r][rr] & & \Omega\Phi(\sigma^\ast)^2 X\ar[ur]\ar[rr] & & 0\ar[ur]\\
\Omega^2\Phi(\sigma^\ast)^2X\ar[rr]\ar[ur]& & 0\ar[ur] & & & &
}}
\end{equation}

Up to the precise location of $\sigma^\ast$, we obtain all elements that appear in Diagram~\ref{dia:waffle}. Our final step of this construction is to restrict to all elements with $(0,0)$ in the $V^{\geq 0}$ coordinate. This has underlying diagram
\begin{equation}\label{dia:wafflecone}
\vcenter{\xymatrix@R=1em@C=1em{
\Phi X\ar[r]\ar[dr]&\Phi\Omega\sigma^\ast X\ar[r]\ar[d]&\Phi\Omega^2(\sigma^\ast)^2 X\ar@.[r]\ar[d]&\\
&\Omega\Phi\sigma^\ast X\ar[r]\ar[dr]&\Omega\Phi\Omega(\sigma^\ast)^2 X\ar@.[r]\ar[d]&\\
&&\Omega^2\Phi(\sigma^\ast)^2 X\ar@.[dr]\ar@.[r]&\\
&&&&
}}
\end{equation}
where now the horizontal direction is $\N_1$ and the vertical direction $\N_2$. For example, the three objects in Diagram~\ref{dia:extradimensionafter} in the $(0,0)$ column form the third column of the above diagram. Call this triangular diagram shape $U$, as Heller does, and let $\mathbf X=(0,0)^\ast b_{2,\ast}b_1^\ast\Phi G(X)\in\Sp\D^U$.

The diagonal in Diagram~\ref{dia:wafflecone} is homotopy final. That is, let $d\colon\N\to U$ be the inclusion of the diagonal. Then for any $\mathbf{Y}\in \Sp\D^U$, \eg the diagram $\mathbf{X}$ we have constructed above, the canonical map
\begin{equation}\label{eq:map1}
\pi_{\N,!} d^\ast \mathbf{Y}\overset{\cong}\longrightarrow \pi_{U,!} \mathbf{Y}
\end{equation}
is an isomorphism. We use \cite[Corollary~3.13]{GroPonShu14} to prove this. Let us label the objects of $U$ as indicated by Diagram~\ref{dia:augmentedwaffle} by $(p,q)$. Then we need to show that the slice category $d_{(p,q)/}$ is homotopy contractible for any $(p,q)\in U$. Avoiding the technical definition, it suffices to show that these categories admit an initial object. But this is obvious: there is a unique point on the diagonal closest to any $(p,q)\in U$, namely $(\max\{p,q\},\max\{p,q\})$; this object with the unique map from $(p,q)$ is the initial object of this category slice category, so $d$ is homotopy final.

In our case, $d^\ast\mathbf{X}\cong F\Phi^V(X)$, so $\pi_{\N,!}d^\ast\mathbf{X}\cong\loc_\E\Phi^V(X)$, precisely the domain of the map we are constructing, Equation~\ref{eqn:hardmodification}.

We now claim that the original Diagram~\ref{dia:waffle} can be found in $U$ as well. To see this, we stretch the shape $U$:
\begin{equation}\label{dia:augmentedwaffle}
\vcenter{\xymatrix@R=1em{
\Phi X^0\ar[r]\ar[d]& \Phi X^1\ar@{-->}[d]\ar[r]\ar[dl]&\Phi X^2\ar[r]\ar@{-->}[d]\ar[dl]&\cdots\\
\Omega\sigma^\ast\Phi X^0\ar[r]\ar[d]& \Omega\sigma^\ast\Phi X^1\ar@{-->}[d]\ar[r]\ar[dl]&\Omega\sigma^\ast\Phi X^2\ar[r]\ar@{-->}[d]\ar[dl]&\cdots\\
\Omega^2(\sigma^\ast)^2\Phi X^0\ar[r]\ar[d]& \Omega^2(\sigma^\ast)^2\Phi X^1\ar@{-->}[d]\ar[r]&\Omega^2(\sigma^\ast)^2\Phi X^2\ar[r]\ar@{-->}[d]&\cdots\\
\vdots&\vdots&\vdots
}}
\end{equation}
The first row of Diagram~\ref{dia:augmentedwaffle} corresponds to the first row of Diagram~\ref{dia:wafflecone} and the first column corresponds to the diagonal. The vertical arrows in Diagram~\ref{dia:wafflecone} thus have stretched to the diagonal arrows in Diagram~\ref{dia:augmentedwaffle}, and the dashed vertical arrows are just compositions. In this way, we can see $\N^2$ is a non-full subcategory of $U$, and we let $i_2\colon \N^2\to U$ be the inclusion. This gives in total
\begin{equation*}
\xymatrix@C=1em{
\N\ar@/^1pc/[rr]^-{d}\ar[r]_{i_1}&\N^2\ar[r]_{i_2}&U,
}
\end{equation*}
where $\N\to\N^2$ is the map into the first column and $\N^2\to U$ is the map indicated by the above diagram. The composition is the diagonal map $d\colon\N\to U$.

We now claim that $i_2$ is also a homotopy final functor. It suffices to prove that each of the categories ${i_2}_{(p,q)/}$ admits an initial object for any $(p,q)\in U$. Because $\N^2\to U$ is surjective (though not full), we have that $(p,q)=i_2(p,q)$ is an element of each ${i_2}_{(p,q)/}$, and this object is initial. Thus we have an isomorphism
\begin{equation}\label{eq:map2}
\pi_{\N^2,!}i_2^\ast\mathbf{Y}\overset{\cong}\longrightarrow\pi_{U,!}\mathbf{Y}
\end{equation}
for any $\mathbf{Y}\in\Sp\D^U$. Using our constructed $\mathbf{X}$, we have $\pi_{\N^2,!}i_2^\ast\mathbf{X}\cong\pi_!\pi_! F\Phi F(X)$, which is the codomain of Equation~\ref{eqn:hardmodification}. Putting these two isomorphisms together, we have the required isomorphism
\begin{equation*}
\vcenter{\xymatrix@C=1em{
\loc_\E\Phi^V(X)\cong\pi_{\N,!}d^\ast \mathbf{X}\ar@(d,d)[rr]_{\loc\Phi^V\eta_{\D,X}}\ar[r]^-\cong&\pi_{!,U}\mathbf{X}&\ar[l]_-\cong\pi_{\N^2,!}i_2^\ast\mathbf{X}\cong\loc_\E\Phi^Vi_\D\loc_\D(X).
}}
\end{equation*}
This completes the proof.
\end{proof}

\begin{notn}
Let $\textbf{Der}_!$ be the 2-category with objects regular pointed derivators, cocontinuous morphisms, and all modifications. Let $\textbf{StDer}_!\subset\textbf{Der}_!$ be the full 2-subcategory of stable derivators.
\end{notn}

\begin{prop}
The assignment $\D\mapsto\St\D$ and $\Phi\mapsto \St\Phi$ gives a pseudofunctor
\begin{equation*}
\St\colon\textbf{Der}_!\longrightarrow\textbf{StDer}_!.
\end{equation*}
\end{prop}
\begin{proof}
First, we need to show that for $\Phi\in\Hom_{!}(\D,\E)$, $\St\Phi$ is cocontinuous. To see this, let $u\colon J\to K$ be any functor in $\Dia$. Using the natural isomorphism of Lemma~\ref{lemma:stableequivalence}, we know that for any $X\in \Sp\D(J)$,
\begin{equation*}
\loc_\E(\Phi^V(u_!X))\cong\St\Phi(\loc_\D(u_!X)).
\end{equation*}
But $\loc_\E\Phi^V$ is a composition of cocontinuous morphisms, so
\begin{equation*}
\loc_\E(\Phi^V(u_!X))\cong u_!\loc_\E(\Phi^V(X)).
\end{equation*}
Using again the natural isomorphism,
\begin{equation*}
u_!\loc_\E(\Phi^V(X))\cong u_!\loc_\D(\St\Phi(X)).
\end{equation*}
Because $\loc_\D$ is cocontinuous, it also commutes with $u_!$. Putting this all together we conclude that
\begin{equation*}
\St\Phi(u_!\loc_\D(X))\cong u_!\St\Phi(\loc_\D(X)).
\end{equation*}
Therefore $\St\Phi$ preserves left Kan extensions along $u$ on the essential image of $\loc_\D$, which is all of $\St\D$.

To show that $\St$ is a pseudofunctor, we also need to give a natural isomorphism $\id_{\St\D}\to\St(\id_\D)$ for any $\D$. But by definition, $\St(\id_\D)=\loc_\D i_\D$, which is isomorphic to $\id_{\St\D}$ via the counit of the $(\loc_\D,i_\D)$ adjunction.

The last thing to show for pseudofunctoriality is the assignment on morphisms behaves well with respect to composition. Let $\Phi\colon\D\to\E$ and $\Psi\colon\E\to\F$. Then we would like to say that $\St\Psi\St\Phi\cong\St(\Psi\Phi)$ naturally in $\Phi$ and $\Psi$. Unwinding the definitions, we have the following situation:
\begin{equation*}
\xymatrix{
\St\D\ar[r]^{i_\D}&\Sp\D\ar[d]_{\Phi^V}\ar[rr]^{\Psi^V\Phi^V=\,(\Psi\Phi)^V}&&\Sp\F\ar[r]^{\loc_\F}\ar@{}[dll]|\swtrans&\St\F\\
&\Sp\E\ar[r]_{\loc_\E}&\St\E\ar[r]_{i_\E}&\Sp\E\ar[u]_{\Psi^V}
}
\end{equation*}
where the bottom composition is $\St\Psi\St\Phi$ and the top is $\St(\Psi\Phi)$. The natural transformation is induced by the unit $\eta_\E\colon\id_{\Sp\E}\to i_\E\loc_\E$. Specifically, it is $\loc_\F\Psi^V\eta_{\E,\Phi^V i_\D}$. Using essentially the same argument of the preceding lemma, this is an isomodification. This gives us the requisite invertible modification defining composition.

We should technically check associativity of composition of morphisms, but the above constructions show that the modifications involved in the associativity diagrams come from the unit and the counit of the adjunction. Therefore all the necessary associativity conditions hold because they do for compositions of modifications.
\end{proof}

\begin{theorem}\label{thm:main2}
The pseudofunctor $\St\colon\textbf{Der}_{!}\to\textbf{StDer}_!$ is left adjoint to the (fully faithful) inclusion $\textbf{StDer}_!\to\textbf{Der}_!$. That is, for any stable derivator $\bS$,
\begin{equation*}
\stab^\ast\colon \Hom_{!}(\St\D,\bS)\to\Hom_{!}(\D,\bS)
\end{equation*}
given by precomposition with the morphism $\stab\colon\D\to\St\D$ is an equivalence of categories.
\end{theorem}
\begin{proof}
It suffices to give an inverse equivalence $F\colon\Hom_{!}(\D,\bS)\to\Hom_!(\St\D,\bS)$. Suppose that $\Phi\colon\D\to\bS$ is a cocontinuous morphism. Then we let $F(\Phi)$ be given by the composition
\begin{equation*}
\xymatrix{
\D\ar[d]_-\Phi\ar[r]&\Sp\D\ar[d]_-{\Phi^V} \ar[r]^{\loc_\D}&\ar@{}[dl]|\netrans\ar@{}[dl]<0.4ex>_(0.55)\cong{}\St\D\ar[d]_-{\St\Phi}\ar@(r,u)[dr]^{F(\Phi)}\\
\bS\ar[r]&\Sp\bS\ar[r]_{\loc_\bS}&\St\bS\ar[r]_-{(0,0)^\ast}&\bS
}
\end{equation*}
Otherwise put, we have $F(\Phi):=(0,0)^\ast\St\Phi$. The middle square commutes up to isomorphism by Lemma~\ref{lemma:stableequivalence}. Since the upper composition is exactly $\stab_\D\colon\D\to~\St\D$, this shows that we have an isomorphism $F(\Phi)\stab_\D\cong (0,0)^\ast\stab_\bS\Phi$. But because $\bS$ is already a stable derivator, we have that $(0,0)^\ast\stab_\bS\cong \id_\bS$, so that $F(\Phi)\stab_\D\cong\Phi$.

For the other direction, suppose that we have a cocontinuous morphism\linebreak $\Psi\colon\St\D\to~\bS$. Then we would like to show that $F(\Psi\stab_\D)\cong\Psi$. These morphisms fit into the following diagram:
\begin{equation}\label{dia:StStD}
\vcenter{\xymatrix@C=4em{
\D\ar[d]_-{\stab_\D}\ar[r]^-{\stab_\D}&\ar@{}[dl]|\netrans\ar@{}[dl]<-1ex>|(0.55)\cong{}\St\D\ar[d]^-{\St\stab_\D}\ar@(r,u)[dr]&{}\\
\St\D\ar[d]_-{\Psi}\ar[r]^-{\stab_{\St\D}}&\ar@{}[dl]|\netrans\ar@{}[dl]<-1ex>|(0.55)\cong{}\St\St\D\ar[d]\ar[d]^-{\St\Psi}\ar[r]^-{(0,0)^\ast}&\St\D\ar[d]^-\Psi\ar@{}[dl]|\netrans\ar@{}[dl]<-1ex>|(0.55)\cong{}\\
\bS\ar[r]_-{\stab_\bS}&\St\bS\ar[r]_{(0,0)^\ast}&\bS
}}
\end{equation}
The rightmost composition is by definition $F(\Psi\stab_\D)$. Tracing the diagram gives us
\begin{equation*}
\Psi\stab_\D\cong F(\Psi\stab_\D)\stab_\D,
\end{equation*}
but this is not quite enough.

Instead, we would like to compare $\St\stab_\D$ and $\stab_{\St\D}$. The first map is defined by
\begin{equation*}
\xymatrix@C=3em{
\St\D\ar[r]^-{i_\D}&\Sp\D\ar[r]^-{L^V}&\Sp(\Sp\D)\ar[r]^-{\loc^V}&\Sp(\St\D)\ar[r]^-{\loc_{\St\D}}&\St(\St\D)
}
\end{equation*}
and the second defined by
\begin{equation*}
\xymatrix@C=3em{
\St\D\ar[r]^-{L_{\St\D}}&\Sp(\St\D)\ar[r]^-{\loc_{\St\D}}&\St(\St\D).
}
\end{equation*}

We would like to show that $\loc^V L^Vi_\D$ and $L_{\Sp\D}i_\D$ are isomorphic after applying $\loc_{\St\D}$. Therefore let $X\in\St\D$. Then $L^V(X)\in\Sp(\Sp\D)\subset \D^{V\times V}$, which has underlying diagram
\begin{equation*}
\vcenter{\xymatrix@C=1em@R=1em{
&{}&{}&{}&{}\\
&&0\ar[r]&LX_1\ar@.[ur]&{}\\
&0\ar[r]&LX_0\ar[r]\ar[u]&0\ar[u]\\
&LX_{-1}\ar[r]\ar[u]&0\ar[u]\\
\ar@.[ur]&
}}
\end{equation*}
where each object is the connective suspension prespectrum on $X_i=(i,i)^\ast X$. Applying $\loc^V$ to each of these gives back the original spectrum $X$ up to shifting because $X$ was originally a stable spectrum, just as in the proof of Proposition~\ref{prop:stabonstable}:
\begin{equation}\label{dia:stabstab}
\vcenter{\xymatrix@C=1em@R=1em{
&{}&{}&{}&{}\\
&&0\ar[r]&\sigma^\ast X\ar@.[ur]&{}\\
&0\ar[r]&X\ar[r]\ar[u]&0\ar[u]\\
&(\sigma\inv)^\ast X\ar[r]\ar[u]&0\ar[u]\\
\ar@.[ur]&
}}
\end{equation}
Call this object $\overline X$. Already we have that $\overline X\to \Omega\sigma^\ast \overline X$ is an isomorphism, \ie $\overline X$ is a stable spectrum in $\St\D$, so $\loc_{\St\D} \overline X\cong\overline X$.

Now, consider the same $X\in\St\D$ but now apply $L_{\St\D}$. This is a connective suspension prespectrum on $X$, which has underlying diagram
\begin{equation*}
\vcenter{\xymatrix@C=1em@R=1em{
&{}&{}&{}&{}\\
&&0\ar[r]&\Sigma X\ar@.[ur]&{}\\
&0\ar[r]&X\ar[r]\ar[u]&0\ar[u]\\
&0\ar[r]\ar[u]&0\ar[u]\\
\ar@.[ur]&
}}
\end{equation*}
To analyze $\widetilde X=\loc_{\St\D}L_{\St\D} X$, we can look at $(i,i)^\ast\widetilde X$, taking this restriction on the objectwise $V$ dimension. Namely, $(i,i)^\ast\widetilde X\in\St\D$ is the localization of
\begin{equation*}
\vcenter{\xymatrix@C=1em@R=1em{
&{}&{}&{}&{}\\
&&0\ar[r]&\Sigma X_i\ar@.[ur]&{}\\
&0\ar[r]&X_i\ar[r]\ar[u]&0\ar[u]\\
&0\ar[r]\ar[u]&0\ar[u]\\
\ar@.[ur]&
}}
\end{equation*}
Again, because $X$ was originally a stable spectrum, this is just $\loc_\D LX_i$, which is isomorphic to $(\sigma^\ast)^iX$. Therefore $\widetilde X$ is isomorphic to $\overline X$.

But now we use the fact that $(0,0)^\ast$ is a quasiinverse for $\stab_{\St\D}$, \ie the middle composition of Diagram~\ref{dia:StStD}. This makes $(0,0)^\ast$ a quasiinverse for $\St\stab_\D$ as well, so that the top-right triangular horizontal composition is isomorphic to the identity:
\begin{equation*}
\vcenter{\xymatrix@C=4em{
\St\D\ar[d]_-{\St\stab_\D}\ar@(r,u)[dr]^-=&\ar@{}[dl]|\netrans\ar@{}[dl]<-1ex>|(0.55)\cong{}\\
\St\St\D\ar[d]\ar[d]_-{\St\Psi}\ar[r]^-{(0,0)^\ast}&\St\D\ar[d]^-\Psi\ar@{}[dl]|\netrans\ar@{}[dl]<-1ex>|(0.55)\cong{}\\
\St\bS\ar[r]_{(0,0)^\ast}&\bS
}}
\end{equation*}
This gives the required isomorphism $F(\Psi\stab_\D)=(0,0)^\ast\St\Psi\St\stab_\D\to \Psi$ and completes the proof of \cite[Theorem~8.1]{Hel97}.
\end{proof}

\bibliography{Bibliography}
\bibliographystyle{alpha}

\end{document}